\documentclass[11pt,leqno,a4paper]{amsart}
\usepackage{amssymb,amsmath,amscd,graphicx,fontenc,amsthm,mathrsfs, mathtools}
\usepackage[frame, cmtip, arrow, matrix, line, graph, curve]{xy}
\usepackage{graphpap, color}
\usepackage[mathscr]{eucal}
\usepackage{ifthen}
\usepackage{relsize}

\usepackage{fancyhdr}
\usepackage{indentfirst}
\usepackage{paralist}
\usepackage{pstricks}
\usepackage{bbm}
\usepackage{tikz}

\newtheorem{theorem}{Theorem}[section]

\newtheorem{lemma}[theorem]{Lemma}
\newtheorem{proposition}[theorem]{Proposition}
\newtheorem{corollary}[theorem]{Corollary}
\newtheorem{notation}[theorem]{Notation}

\theoremstyle{remark}
\newtheorem{remark}[theorem]{Remark}

\numberwithin{equation}{section}
\newcommand{\NN}{{\mathbb{N}}}
\newcommand{\ZZ}{{\mathbb{Z}}}

\newcommand{\RR}{{\mathbb{R}}}
\newcommand{\PP}{{\mathbb{P}}}
\newcommand{\QQ}{{\mathbb{Q}}}

\renewcommand{\SS}{{\mathbb{S}}}
\renewcommand{\PP}{{\mathbb{P}}}
\DeclareMathOperator{\rk}{rk}

\DeclareMathOperator{\covol}{covol}
\DeclareMathOperator{\supp}{supp}

\newcommand{\diag}{\mathrm{diag}}

\newcommand{\GL}{\mathrm{GL}}
\newcommand{\SL}{\mathrm{SL}}

\newcommand{\UL}{\mathrm{UL}}

\newcommand{\SO}{{\mathrm{SO}}}

\newcommand{\SG}{\mathsf{G}}

\newcommand{\sg}{\mathsf{g}}
\newcommand{\sh}{\mathsf{h}}

\newcommand{\T}{{\mathsf{T}}}
\newcommand{\CJ}{\mathsf{J}}

\renewcommand{\t}{\mathsf{t}}
\renewcommand{\j}{\mathsf{j}}

\newcommand{\I}{\mathsf{I}}

\newcommand{\q}[2]{\mathsf q^{#1}_{#2}}

\newcommand{\ox}{\mathbf x}

\newcommand{\ow}{\mathbf w}
\newcommand{\ov}{\mathbf v}

\providecommand{\ve}{\mathbf{ e}}

\newcommand{\NT}{\mathbf N}
\newcommand{\VT}{\mathbf V}

\newcommand{\Siegel}[2]{\mathcal S_{#1}(#2)}
\newcommand{\id}{\mathrm{I}}
\newcommand{\MV}{\mathrm{M}}
\renewcommand{\varpi}{\pi}%{\clubsuit}
\newcommand{\pp}{\mathtt p}
\newcommand{\origin}{O}
\newcommand{\one}{{\mathsf 1}}

\makeatletter
\def\moverlay{\mathpalette\mov@rlay}
\def\mov@rlay#1#2{\leavevmode\vtop{%
   \baselineskip\z@skip \lineskiplimit-\maxdimen
   \ialign{\hfil$\m@th#1##$\hfil\cr#2\crcr}}}
\newcommand{\charfusion}[3][\mathord]{
    #1{\ifx#1\mathop\vphantom{#2}\fi
        \mathpalette\mov@rlay{#2\cr#3}
      }
    \ifx#1\mathop\expandafter\displaylimits\fi}
\makeatother

\newcommand{\bigcupdot}{\charfusion[\mathop]{\bigsqcup}{\cdot}}

\definecolor{cmd}{rgb}{1.0, 0.35, 0.21}

\begin{document}
\title[Mean Value of $S$-arithmetic Siegel transform]{Rogers' mean value theorem for $S$-arithmetic Siegel transforms and applications to the geometry of numbers}
\author{Jiyoung Han}

\subjclass[2010]{11H60, 11P21, 37A45}

\maketitle
\begin{abstract}
We prove higher moment formulas for Siegel transforms defined over the space of unimodular $S$-lattices in $\QQ_S^d$, $d\ge 3$, where in the real case, the formulas are introduced by Rogers \cite{Rogers55}. 
As applications, we obtain the random statements of Gauss circle problem for any star-shaped sets in $\QQ_S^d$ centered at the origin and of the effective Oppenheim conjecture for $S$-arithmetic quadratic forms.
\end{abstract}

%\tableofcontents

%\subsection*{Acknowledgments} 

%%%%%%%%%%%%%%%%%%%%%%%%%%%%%%%%%%%%%%%%%%%%%%%%%%%%%%%%%%%%%%%%%%%%%%%%%%%%%%%%%%%%%%%%%%%%%%%%%%%%%%%%%%
\section{Introduction}
% We will define $\widetilde f$ and $\Siegel{1}{f}$ differently, excluding/including $f(\origin)$, respectively.

%%%%%%%%%%
Let $G=\SL_d(\RR)$ and $\Gamma=\SL_d(\ZZ)$. Consider the homogeneous space $G/\Gamma$, which is identified with the space of unimodular lattices $\Lambda$ in $\RR^d$. 
For any bounded and compactly supported function $f : \RR^d \rightarrow \RR_{\ge0}$, one can define \emph{the Siegel transform} $\widetilde f$ of $f$ by
\begin{equation}\label{eq Siegel}
\widetilde f (\Lambda) = \sum_{\ov \in \Lambda-\{\origin\}} f(\ov), \;\forall \Lambda \in G/\Gamma,
\end{equation}
where $\origin=(0,\ldots,0)$ is the origin in $\RR^d$.

%for any $\Lambda \in \SL_d(\RR)/\SL_d(\ZZ)$.
There is the famous integral formula by Siegel \cite{Sie98}; %between $f$ and $\widetilde f$
\[
\int_{G/\Gamma} \widetilde f(g\ZZ^d) dg = \int_{\RR^d} f(\ov) d\ov,
\]
where $dg$ is the normalized $G$-invariant measure of $G/\Gamma$ and $d\ov$ is the usual Lebesgue measure of $\RR^d$.

%%%%%%%%%%
In 1955, Rogers~\cite{Rogers55} provided the $k$-th moment formulas for the Siegel transform, where $2\le k\le d-1$. In particular, the second moment formula is as follow (\cite[Theorem 4]{Rogers55}. See also \cite[Lemma 4]{Rogers56-2}):
\[
\int_{\SL_d(\RR)/\SL_d(\ZZ)} {\widetilde f(g\ZZ^d)}^2 dg
=\left(\int_{\RR^d} f(\ov) d\ov \right)^2 
+\hspace{-0.2in}\sum_{\scriptsize \begin{array}{c}k\in \NN, q\in \ZZ-\{0\}\\ \gcd(k,q)=1\end{array}}
 \hspace{-0.09in}\int_{\RR^d} f(k\ov) f(q\ov) d\ov.
\]

Let us briefly explain the core property that Rogers used to compute moment formula as follows.
Consider the diagonal flow $$a_t=\diag(e^{-t/(d-1)}, \ldots, e^{-t/(d-1)}, e^t)$$ in $G$ and $$H^-=\{g\in G: a_t^{-1}ga_t\rightarrow \id_d \text{ as }t\rightarrow \infty\}.$$
For $\phi:G/\Gamma\rightarrow \RR_{\ge 0}$, define
\[
\MV(\phi)=\lim_{t\rightarrow \infty}
\int_{H^-/(H^-\cap \Gamma)} \phi(a_t h\Gamma) dh,
\]  
where $dh$ is the normalized Haar measure on $H^-/(H^-\cap \Gamma)$.
Let $L_g$ be the left multiplication by $g\in G$. If $\MV(\phi\circ L_g)=\MV(\phi)$ for almost all $g\in G$, then
\begin{equation}\label{eq 1: introduction}
\int_{G/\Gamma} \phi\: dg=M(\phi).
\end{equation}

He deduced \eqref{eq 1: introduction} using the property that
\begin{equation}\label{eq 2: introduction}
\int_{G/\Gamma} \phi(gg'\Gamma)dg
=\int_{G/\Gamma} \phi(g\Gamma) dg,\; \forall g'\in G
\end{equation}
(see \cite[Page 256]{Rogers55}) which is not true in general. It is unknown that the above equality \eqref{eq 2: introduction} is true even for $\phi=\widetilde f\:^k$, where $f$ is a non-negative, bounded and compactly supported function on $\RR^d$.

%%%%
The limit equation \eqref{eq 1: introduction} for bounded functions is followed from the special case of Theorem 1.4 in \cite{Shah1996} (see \cite{Gui2010} for the $S$-arithmetic space). However, it is difficult to extend \eqref{eq 1: introduction} to unbounded functions $\phi$ such as powers $\widetilde f\:^k$ of the Siegel transform.

%%%%
The first goal of this paper is to give an alternative proof of the $k$-th moment formulas for the Siegel transform in more general setting- for the $S$-arithmetic space $\QQ_S^d$. Note that the simplest $S$-arithmetic space is the real space $\RR^d$.

%%%%%%%%%%
Based on these integral formulas, it is possible to approach many number-theoretic problems related to lattices in $\RR^d$, 
using homogeneous dynamics. 
For instance, for a given nondegenerate isotropic irrational quadratic form $q$ on $\RR^d$, 
Dani and Margulis \cite{DM} and Eskin, Margulis and Mozes \cite{EMM} used the Siegel transform to obtain the asymptotic limit of the counting function
\[
\NT(q, a, b, T)=\#\left\{\ov \in \ZZ^d : a< q(\ov)< b,\; \|\ov\|< T \right\}.
\] 
%which is called a quantitative Oppenheim conjecture.

%%%%%%%%%%
Using the Rogers' second moment theorem, 
Athreya and Margulis \cite{AM18} showed that there is a positive $\delta>0$ such that
for almost all isotropic quadratic form $q$ of rank at least 3, 
the difference between $\NT(q, a, b, T)$ and its asymptotic limit is $o(T^{d-2-\delta})$. 
Kelmer and Yu \cite{KY} obtained similar results for specific homogeneous polynomials of arbitrary even degree and Bandi, Ghosh and the author \cite{BGH} generalized these results for pairs of a specific homogeneous polynomial of even degree and a linear map.
Ghosh, Kelmer and Yu \cite{GKY} extended the first and the second moment theorems for congruence subgroups and obtained the result about the error term between $\NT(q^{}_\xi, a, b, T)$ and its asymptotic limit for generic inhomogeneous quadratic forms $q^{}_\xi$, for any fixed shift $\xi\in\RR^d$.
See also \cite{Sar} and \cite{GGN} for systems of linear maps on a quadratic surface, and \cite{HLM} for the $S$-arithmetic set-up.

%%%%%%%%%%
There are also applications to problems related to Diophantine approximation.
%%%
Kleinbock and Margulis \cite{KM99} used the Siegel transform (defined over a subset of primitive vectors) to obtain the Borel-Cantelli family,
which is connected to the study of Diophantine approximation of systems of $m$ linear forms in $n$ variables.
%%% CLT diophantine 
Bj$\ddot{\text{o}}$rklund and Gorodnik \cite{BG17,BG18} showed central limit theorems for 
the lattice counting problem and the matrix version of Diophantine approximation, respectively.
See also \cite{DFV17} for the central limit theorem of Diophantine approximation for another generalization to the case of higher dimension.
%%%

%%%%%%%%%%
The second goal of this paper is that as applications, combining with the results of \cite{AM09} and \cite{HLM}, 
we show the random statement of effective version of Oppenheim conjecture for $S$-arithmetic quadratic forms and 
of the Gauss circle problem for arbitrary star-shaped sets centered at the origin, for the $S$-arithmetic set-up.
These results are generalizations of the works of Athreya and Margulis \cite{AM18} and the special case of the result of Schmidt \cite{Sch60}.

%%%%%%%%%%
\subsection*{Organization}
%%% organization
In Section 2, we state our main theorems and applications which are proved in Section 3 and Section 5.
In Section 3, to prove higher moment formulas, we introduce the $\alpha$-function defined on the space of unimodular lattices (see the definition \eqref{alpha function}) which is useful to prove the integrability of unbounded functions similar to the Siegel transform (see the definition \eqref{widehat function}).
In Section 4, we first give an example which explains why Lemma 2.1 in \cite{AM18} fails in the $S$-arithmetic space, when $S\neq \{\infty\}$. And then for applications in Section 5, we provide examples of families of bounded measurable sets in $\QQ_S^d$, which satisfy the statement of Lemma 2.1 in \cite{AM18}.

%%%%%%%%%%
\subsection*{Acknowledgments}
I would like to thank Seonhee Lim and Anish Ghosh for valuable advices and discussion.
I am also grateful to Jayadev Athreya, Keivan Mallahi-Karai, Seungki Kim, Dmitry Kleinbock and Mishel Skenderi.
This paper is supported by the Samsung Science and Technology Foundation under project No. SSTF-BA1601-03 and the National Research Foundation of Korea (NRF) grant funded by the Korea government under project No. 0409-20200150.

%%%%%%%%%%%%%%%%%%%%%
% Section %%%%%%%%%%%
%%%%%%%%%%%%%%%%%%%%%
\section{Statement of Results}\label{Statement of Results}

\subsection*{$S$-arithmetic space}
%%%%%%%%%%
Let $S$ be a union of $\{\infty\}$ and a finite set $S_f=\{p_1, \ldots, p_s\}$ of primes.
If we refer $p\in S$, it includes the case that $p=\infty$. 
Consider the completion field $\QQ_p$ of $\QQ$ with respect to the $p$-adic norm $|\cdot|_p$ and let $\QQ_\infty=\RR$. 
Define $\QQ_S=\prod_{p\in S} \QQ_p$ and 
\[\ZZ_S=\{(z,\ldots, z)\in \QQ_S : z\in \QQ\text{ and } |z|_v \le 1,\;\forall v\notin S\}.
\]
%If $p<\infty$, let $\ZZ_p=\{z \in \QQ_p : |z|_p \le 1\}$.
The set $\ZZ_S$ is called \emph{the ring of $S$-integers}.
Note that if $S=\{\infty\}$, then $\QQ_S=\RR$ and $\ZZ_S=\ZZ$.

Let us assign the measure $\mu=\prod_{p\in S} \mu_p$ on $\QQ_S$, where $\mu_p$ is the Haar measure on $\QQ_p$, $p\in S$.
If $p<\infty$, we normalize $\mu_p$ so that $\mu_p(\ZZ_p)=1$, where $\ZZ_p=\{x \in \QQ_p : |x|_p\le 1\}$. 
For simplicity, we denote the product measure $\mu^d$ of $\QQ_S^d$ by $\mu$ and we regard $\mu$ as the volume of $\QQ_p^d$. 

We will choose the supremum norm on $\RR^d$ instead of the usual Euclidean norm,
and define the norm $\|\cdot\|_S$ on $\QQ_S^d$ by
\[\|\ov\|_S=\max_{p\in S} \|\ov_p\|_p,
\] 
where $\ov=(\ov_p)_{p\in S}$.

A free $\ZZ_S$-module $\Lambda$ of rank $d$ in $\QQ_S^d$ is a uniform lattice subgroup of the additive group $\QQ_S^d$, 
that is, $\QQ_S^d/\Lambda$ is compact and hence has a finite $\QQ_S^d$-invariant measure.
We will call $\Lambda$ \emph{an $S$-lattice} of $\QQ_S^d$.

Define $\UL_d(\QQ_S)$ by
\[\UL_d(\QQ_S)=\left\{\sg=(g_p)_{p \in S} \in \GL_d(\QQ_S) : \begin{array}{c}\det g_\infty=1\;\text{and}\\ \left|\det g_p\right|_p=1, \;p\in S_f.\end{array} \right\}
\]
and $\UL_d(\ZZ_S)=\UL_d(\QQ_S)\cap \mathfrak M_{d}(\ZZ_S)$. 
Here, $\mathfrak M_{d}(\ZZ_S)=\mathfrak M_{d,d}(\ZZ_S)$, where $\mathfrak M_{m,n}(K)$ is the space of $m\times n$ matrices whose entries are in $K$.
As $\SL_d(\RR)/\SL_d(\ZZ)$ is identified with the space of unimodular lattices in $\RR^d$, 
$\UL_d(\QQ_S)/\UL_d(\ZZ_S)$ is identified with the space of unimodular $S$-lattices in $\QQ_S^d$ by the map
%Then it is easy to check that the space of unimodular lattices in $\QQ_S^d$ is identified with $\UL_d(\QQ_S)/\UL_d(\ZZ_S)$ by the map
\[\sg \in \UL_d(\QQ_S) \mapsto \sg\ZZ_S^d \subset \QQ_S^d.
\] 
Note that $\SL_d(\QQ_S)/\SL_d(\ZZ_S)$, where $\SL_d(\QQ_S)=\prod_{p\in S} \SL_d(\QQ_p)$, is properly embedded to $\UL_d(\QQ_S)/\UL_d(\ZZ_S)$.

\begin{remark}
Following \cite{HLM}, we will use sans serif typestyle for parameters of an $S$-arithmetic space.
In particular, we will denote by $\T=(T_\infty, T_{p_1}, \ldots, T_{p_s})$ a radius vector. It is an element of $\RR_{\ge0}\times \prod_{p\in S_f} p^{\ZZ}\subseteq (\RR_{\ge 0})^{s+1}$. 
We always regard that there is a relation between $\T$ and $\t=(t_p)_{p\in S}$ as follows: $T_\infty=e^{t_\infty}$ and $T_p=p^{t_p}$, $p\in S_f$.
We say that $\T=(T_p)_{p\in S}\succeq \T'=(T'_p)_{p\in S}$ if $T_p\ge T'_p$, $\forall p\in S$ and
$\T\rightarrow \infty$ if $T_p\rightarrow \infty$, $\forall p\in S$. 
Notions $\t\succeq \t'$ and $\t\rightarrow\infty$ are defined in the similar way.
\end{remark}

\begin{remark}\label{Remark 2.2}
Throughout the paper, if we denote by $\SG/\Gamma$, it is either $\SG=\UL_d(\QQ_S)$ and $\Gamma=\UL_d(\ZZ_S)$ or $\SG=\SL_d(\QQ_S)$ or $\Gamma=\SL_d(\ZZ_S)$.
Let us denote by $m=m_{\SG}$ either a Haar measure on a group $\SG$ or a $\SG$-invariant probability measure on $\SG/\Gamma$.
For simplicity, we will use $d\ov$ and $d\sg$ instead of $d\mu(\ov)$ and $dm(\sg)$, respectively.
\end{remark}

%%%%%%%%%%%%%%%%%%%%%%%%%%
%%%%%%%%%%%%%%%%%%%%%%%%%%
\subsection*{Main results}
For a bounded and compactly supported function $f$ on $\QQ_S^d$,
as in \eqref{eq Siegel}, 
define \emph{the Siegel transform} $\widetilde f$ on $\SG/\Gamma$ by
\[
\widetilde f (\Lambda) = \sum_{\ov \in \Lambda-\{\origin\}} f(\ov),
\]
where $\origin=(0,0,\ldots,0)$ is the origin of $\QQ_S^d$.

As in the real case, we have the following integral formula for the $\S$-arithmetic Siegel transform (\cite[Proposition 3.11]{HLM}).
\begin{proposition}\label{HLM Proposition 3.11} 
Let $d\ge 2$.
For a continuous bounded function $f:\QQ_S^d\rightarrow \RR_{\ge 0}$ of compact support, it follows that
%its Siegel transform $\widetilde{f}$ is $\mathcal L^r$ for $1\le r < d$ and
\begin{equation}\label{eq Siegel transform}
\int_{\SL_d(\QQ_S)/\SL_d(\ZZ_S)} \widetilde{f} d\sg = \int_{\QQ_S^d} f d\ov.
\end{equation}
In other words, the $\mathcal L^1$-norm of $\widetilde f$ defined on $\SL_d(\QQ_S)/\SL_d(\ZZ_S)$ is equal to the $\mathcal L^1$-norm of $f$ defined on $\QQ_S^d$.
\end{proposition}

We want to extend this integral formula to the case of the $\mathcal L^k$-norm of $\widetilde f$, where $1\le k\le d-1$ and $\SG/\Gamma=\UL_d(\QQ_S)/\UL_d(\ZZ_S)$ or $\SL_d(\QQ_S)/\SL_d(\ZZ_S)$, for the function $f$ as in Proposition~\ref{HLM Proposition 3.11}.

\vspace{0.1in}
Fix a positive integer $k\le d$. % original condition : $k\le n$, $n \le d-1$.
Let us define the transform $\Siegel{k}{\cdot}$ 
from the space of bounded functions $f$ on $(\QQ_S^d)^k$ of compact support 
to the space of functions on $\SG/\Gamma$ by
%For a bounded function $f:(\QQ_S^d)^k\rightarrow \RR_{\ge0}$ of compact support,  
%define \emph{a generalized Siegel transform} $\Siegel{k}{f}$ of $f$ by
\begin{equation}\label{Siegel k}
\Siegel{k}{f}(\Lambda)=\sum_{\ov_1, \ldots, \ov_k\in \Lambda} f(\ov_1, \ldots, \ov_k),\;\forall \Lambda \in \SG/\Gamma.
\end{equation}

Notice that for $f:\QQ_S^d\rightarrow \RR_{\ge0}$, the summation $\Siegel{k}{f^{(k)}}$, where $f^{(k)}(\ov_1, \ldots, \ov_k):=f(\ov_1)\cdots f(\ov_k)$, is taken over the set $\Lambda^k$,
whereas ${\widetilde f\:}^k$ is taken over $(\Lambda-\{\origin\})^k$. As you can see in Section~\ref{Proof of main theorem}, one can easily obtain the formula of $\int_{\SG/\Gamma} {\widetilde f\:}^k d\sg$ from the proof of Theorem~\ref{MV Theorem 4} (see Remark~\ref{sum over nonzero vectors}). 

\vspace{0.1in}
%%%%%%%%%%
For the computation of the mean value of $\Siegel{k}{f}$ over $\SG/\Gamma$,
we need some notations.
%%%%%%%%%%
We say that a matrix $D\in \mathfrak M_{r,k}(\ZZ_S)$ is \emph{reduced} if there is an integer $t\in \NN$ such that $tD\in \mathfrak M_{r,k}(\ZZ)$ and $\gcd\{(tD)_{ij}:1\le i\le r, 1\le j\le k\}=1$.

\begin{notation} \label{notations}
\begin{enumerate}
\item For a set $S=\{\infty, p_1, \ldots, p_s\}$, denote %$\NN_S$ the set of positive integers which are coprime to each $p_i$, $i=1,\ldots,s$:
\[\begin{split}
\NN_S&=\{q \in \NN : \gcd(q, p_1\cdots p_s)=1\}\;\text{and}\\
\PP_S&=\{p_1^{m_1}\cdots p_s^{m_s} : m_1, \ldots, m_s \in \ZZ_{\ge 0}\}
\end{split}\]
so that $\NN=\NN_S\cdot\PP_S$.
If $S=\{\infty\}$ and $\QQ_S=\RR$, we set 
\[\NN_S=\NN \quad\text{and}\quad\PP_S=\{1\}.\] 

\item 
For $q\in \NN_S$ and $r\in\{1,\ldots,k\}$, define $\mathfrak D_{r,q}$ by the set of reduced matrices $D$ in $\mathfrak M_{r,k}(\ZZ_S)$ for which 
$\exists 1\le j_1<j_2<\cdots<j_r\le k$ such that
\begin{enumerate}[(i)]
\item $\left([D]^{j_1}, [D]^{j_2}, \ldots, [D]^{j_r}\right)=q\id_r;$
\item $D_{ij}=0 \text{ for } 1 \le i \le r \text{ and } 1 \le j<j_i,$
\end{enumerate}
where $[D]^j$ is the $j$-th column of $D$.

\item For a positive $q$ and an $r\times k$ matrix $D$, let $N(D,q)$ be the number of vectors $\ox \in \{0,1,\ldots, q-1\}^r$ for which
\[\frac 1 q \ox D \in \ZZ_S^k.
\]
\end{enumerate}
\end{notation}

%%%%%%%%%%%%%%%%%%%%%%%%%%%%%%%%%%%
\begin{theorem}\label{MV Theorem 4}
Let $d\ge 2$ and $1\le k\le d-1$.
Let $f : (\QQ_S^d)^k \rightarrow \RR_{\ge 0}$ be a continuous bounded function of compact support and $\Siegel{k}{f}$ be defined as in \eqref{Siegel k}. Then $\int_{\SG/\Gamma} \Siegel{k}{f} d\sg <\infty$ and
\begin{equation*}\begin{split}%\label{eq MV Theorem 4 1}
\int_{\SG/\Gamma} \Siegel{k}{f} d\sg&= f(\origin, \ldots, \origin)+\int_{(\QQ_S^d)^k} f(\ov_1, \ldots, \ov_k) d\ov_1\cdots d\ov_k\\
&\hspace{-0.2in}+\sum_{r=1}^{k-1} \sum_{q\in \NN_S} \sum_{D\in \mathfrak D_{r,q}} \frac {N(D,q)^d} {q^{dr}}  \int_{(\QQ_S^d)^r} 
f\Big(\frac 1 q\left(\ov_1, \ldots, \ov_r\right)D\Big)d\ov_1 \cdots d\ov_r.
%&+\sum_{(\nu;\mu)}\sum_{q=1}^\infty \sum_{D} ftn(D,q) \int_{(\QQ_S^d)^m} f(\ov_1, \ldots, \ov_r)
\end{split}\end{equation*}
%Moreover, if $k\le d-1$, the above integral is finite.
\end{theorem}

\begin{remark}
In the case that $S=\{\infty\}$ and $\QQ_S=\RR$, the formula in Theorem~\ref{MV Theorem 4} is the same formula in Theorem 4 in \cite{Rogers55} 
since we can compute $N(D,q)=e_1\cdots e_r$, where
 $e_i=\gcd(\varepsilon_i,q)$ ($1\le i \le r$) when $\varepsilon_1, \ldots, \varepsilon_r$ are the elementary divisors of the matrix $D$.
\end{remark}

\begin{remark}\label{sum over nonzero vectors}
Let $\mathfrak D'_{r,q}$ be the subset of $\mathfrak D_{r,q}$ in Notation~\ref{notations} consisting of matrices D for which all columns of $D$ are nontrivial. Then for any $d\ge2$ and $1\le k \le d-1$, for any continuous bounded function $f:\QQ_S^d\rightarrow \RR_{\ge 0}$ of compact support, it follows that $\int_{\SG/\Gamma} {\widetilde f\:}^2 d\sg<\infty$ and 
\begin{equation*}\begin{split}%\label{eq MV Theorem 4 1}
\int_{\SG/\Gamma} {\widetilde f\:}^k d\sg&=\int_{(\QQ_S^d)^k} f(\ov_1, \ldots, \ov_k) d\ov_1\cdots d\ov_k\\
&\hspace{-0.2in}+\sum_{r=1}^{k-1} \sum_{q\in \NN_S} \sum_{D\in \mathfrak D'_{r,q}} \frac {N(D,q)^d} {q^{dr}}  \int_{(\QQ_S^d)^r} 
f\Big(\frac 1 q\left(\ov_1, \ldots, \ov_r\right)D\Big)d\ov_1 \cdots d\ov_r.
%&+\sum_{(\nu;\mu)}\sum_{q=1}^\infty \sum_{D} ftn(D,q) \int_{(\QQ_S^d)^m} f(\ov_1, \ldots, \ov_r)
\end{split}\end{equation*}
\end{remark}

As a corollary, we obtain the second moment theorem for Siegel transforms.

%%%%%%%%%%%%%%%%%%%%%%%%%%%%%%%%%%%%%%%%%%
\begin{corollary}\label{Corollary L2 norm}
Let $d\ge 3$.
For $f:\QQ_S^d\rightarrow \RR_{\ge 0}$ be a continuous bounded function of compact support, it holds that
\begin{equation*}\begin{split}
\int_{\SG/\Gamma} 
{\widetilde f\;}^2 d\sg
=\left(\int_{\QQ_S^d} f d\ov\right)^2
%\left(\int_{\SG/\Gamma} \Siegel d\sg-f(0)\right)^2\\
+\sum_{q\in \NN_S} \sum_{\pp \in \PP_S} \hspace{-0.12in}\sum_{\scriptsize \begin{array}{c}w\in \ZZ-\{0\}\\ \gcd(q\pp,w)=1\end{array}}
\int_{\QQ_S^d} f\left(q\ov\right)f\left(\frac w \pp \ov\right) d\ov.
\end{split}\end{equation*} 
\end{corollary}

%%%%%%%%%%%%%%%%%%%%%%%%%%
%%%%%%%%%%%%%%%%%%%%%%%%%%
\subsection*{Applications}
For the rest of this section, let us assume that the set $S_f$ contains only finitely many odd primes.
We first introduce two versions of the effective Oppenheim conjecture for $S$-arithmetic isotropic quadratic forms. 
\emph{A quadratic form} $\q{}{}$ on $\QQ_S^d$ is a collection of quadratic forms $q_p$, which is defined on $\QQ_p^d$ for each $p \in S$. We say that $\q{}{}=\{q_p\}_{p\in S}$ is \emph{nondegenerate} or \emph{isotropic} if each $q_p$, $p\in S$, is nondegenerate or isotropic, respectively.

%%%%%%%%%%%%%%%%%%%%%%%%%%%%%%%%%
\begin{theorem}\label{AM Thm 1.1} 
Let $d\ge 3$. Fix $\xi \in \QQ_S$.
For every $\delta>0$, for a.e. nondegenerate isotropic quadratic form $\q{}{}$,
there are constants $c_{\q{}{}}$, $\varepsilon_0=\varepsilon_0(\q{}{}, \xi)>0$ such that for any positive $\varepsilon<\varepsilon_0$, there is a nonzero $\ov \in \ZZ_S^d$ for which
\[ \|\ov\|_S \le c_{\q{}{}} \varepsilon^{-(\frac 1 {d-2} +\delta)}\quad \text{and}\quad |\q{}{}(\ov)-\xi|_S < \varepsilon.
\]
\end{theorem}
%%%%%%%%%%%%%%%%%%%%%%%%%%%%%%%%%

%%%%%%%%%%
For $d=3$ and $S=\{\infty\}$, Ghosh and Kelmer \cite{GK} showed the similar result with a different setting:
for positive $\eta <1$, consider a sequence $\{(N(k), \delta(k))\}_{k\in \NN}$ of pairs of positive numbers such that $\max(\frac {N(k)}{k^\eta \delta(k)^2}, \frac {N(k)^{3/2}}{k^{\eta}\delta(k)})$ goes to zero as $k$ goes to infinity. 
They showed that for almost every isotropic quadratic form $q$ of rank 3, 
\[
\max_{\xi \le N(k)} 
\min_{\scriptsize \begin{array}{c}
\ov \in \ZZ^d\\
\|\ov\|\le ck
\end{array}} |q(\ov)-\xi|\le \delta(k)
\]
for any sufficiently large $k$. 

\vspace{0.1in}

%%%%%%%%%%
For a quadratic form $\q{}{}$ on $\QQ^d_S$, a bounded subset $\I\subset \QQ_S$ and a radius parameter $\T=(T_p)_{p\in S}$, let
\[\begin{split}
\NT(\q{}{}, \I, \T)&=\#\left\{\ov \in \ZZ_S^d : \q{}{}(\ov)\in \I\text{ and }
%\begin{array}{c}\|\ov\|_\infty< T_\infty;\\ \|\ov\|_{p_i}< T_i,\; i=1,\ldots, s\end{array}
\|\ov\|_p<T_p,\; p\in S\right\}\text{ and}\\
\VT(\q{}{}, \I, \T)&=\mu\left(\left\{\ov \in \QQ_S^d : \q{}{}(\ov)\in \I\text{ and }
%\begin{array}{c}\|\ov\|_\infty< T_\infty;\\\|\ov\|_{p_i}< T_i,\; i=1,\ldots, s\end{array}
\|\ov\|_p<T_p,\; p\in S\right\}\right).
\end{split}\]
%where $\|\ov\|_p=\|v_p\|_p$, $\ov=(v_p)_{p\in S}$.

\begin{theorem}\label{AM Thm 1.2} 
Let $d\ge 3$. There exists $\delta_0=\delta_0(d)>0$ such that for any $\delta\in (0,\delta_0)$ and for any $S$-interval $\I$, it satisfies that 
%\[N(\q, \I, \T)=c_{\q} \mu(\I) \|\T\|^{d-2} + o\left(\|\T\|^{(d-2)-\delta}\right).
%\]
\[\NT(\q{}{}, \I, \T)=\VT(\q{}{}, \I, \T) + o_{\q{}{}}\left(|\T|^{(d-2)-\delta}\right)
\]
for almost every isotropic quadratic form $\q{}{}$ of rank $d$. Here, $|\T|=\prod_{p\in S} T_p$.

%For every $\delta>0$, for almost every indefinite quadratic form $\q$ of rank at least $3$,
%for any $S$-interval $\I$,
%\[N(\q, \I, \T)=c_{\q} \mu(\I) \|\T\|^{d-2} + o(\|\T\|^{(d-2)/2+\delta}).
%\]
\end{theorem}

\vspace{0.1in}
%%%%%%%%%%%%%%%%%%%%%%%%%%%%%%%%%
For an arbitrary $S$-lattice $\Lambda$ and a bounded set $A$ in $\QQ_S^d$, 
let $\NT(\Lambda, A)$ be the cardinality of the intersection $\Lambda \cap A$.

The Gauss circle problem is to find the minimal exponent $d-2+\lambda_d$ for which $\NT(\ZZ^d, tA) - \mu(tA)=O(t^{d-2+\lambda_d})$,
where $A$ is a compact convex subset of $\RR^d$ containing the origin and having some boundary condition.
%%%
If $A$ is a sphere, it is known that $\lambda_d=0$ for $d\ge 4$ and it is conjectured to be $\lambda_d=0$ for $d=2$ and $3$ 
(\cite{Kra88, IKKN06}).  
%%%
For the case of any convex set, Guo \cite{Guo12} provides the smallest $\lambda_d$ as far as known: $\lambda_d=(d^2+3d+8)/(d^3+d^2+5d+4)$ for $d\ge 4$
and $73/158$ for $d=3$.

%%%
We will say that $A\subseteq \QQ_S^d$ is \emph{a star-shaped set centered at the origin} if
\[
A=\left\{\ov=(\ov_p)_{p\in S} \in \QQ_S^d : \begin{array}{c}
\|\ov_\infty\|_\infty \le \rho_\infty(\ov_\infty/\|\ov_\infty\|_\infty);\\
\|\ov_p\|_p \le \rho_p(\|\ov_p\|_p \:\ov_p),\; p\in S_f \end{array} \right\}
\]
for some continuous functions $\rho_\infty : \SS^{d-1}\rightarrow \RR_{>0}$ and $\rho_p : \left(\ZZ_p^d-p\ZZ_p^d\right) \rightarrow \RR_{>0}$ $(p\in S_f)$. By the definition, any star-shaped set is bounded.
%%%%%%%%%%%%%%%%%%%%%%%%%%%%%%%%%
\begin{theorem}\label{AM Thm 1.3} Let $d\ge 3$.
Consider a star-shaped set $A_0\subset \QQ_S^d$ centered at the origin and assume that $\mu(A_0)=1$. 
Let $\delta>0$.
Then for almost every unimodular $S$-lattice $\Lambda$, there is $\T_0\succeq 0$ such that 
\[\left|\NT(\Lambda, \T A_0) - |\T|^d\right| < |\T|^{\frac 1 2 d + \delta},\; \forall \T \succeq \T_0,
\]
where $\T A_0$ is the dilate of $A_0$ by $\T$.
%$\T A=\{(v_\infty, v_1, \ldots, v_s)\in \QQ_S^d : v_\infty/T_\infty \in A_\infty \text{ and } T_i v_i \in A_i,\;\forall i\}$.
%for almost every $S$-lattice $\Lambda$ in $\SL_n(\QQ_S)/\SL_n(\ZZ_S)$. 
\end{theorem}

%%%
In the real case, Schmidt \cite{Sch60} obtained the uniform error bound of $\NT(\Lambda, A) - \mu(A)$, which is given as a function of the volume $\mu(A)$, for a family $\Phi$ of bounded Borel sets in $\RR^d$ such that (i) if $A,\; B\in \Phi$, either $A \subseteq B$ or $B \subseteq A$ and (ii) there exists $S\in \Phi$ with arbitrarily large volume. The Schmidt's theorem may not hold for such a family of bounded Borel sets in $\QQ_S^d$ (see Proposition \ref{Prop class F} (c) and the remark below).

%For the case of random lattices, Schmidt \cite{Sch60} showed that for almost every lattice $\Lambda \subset \RR^d$,
%$$\NT(\Lambda, tA)-\vol(tA)=O(\log t\cdot t^{d/2}\psi(\log t)),$$
%where $\psi$ is a positive nondecreasing function such that $\int_0^\infty \psi^{-1}(x) dx <\infty$. 
%
%We modify the Schmidt's theorem to the case of the space of unimodular lattices in $\RR^d$ and extend it to the space of unimodular $S$-lattices.

%%%%%%%%%%%%%%%%%%%%%%%%%%%%%%%%
% section %%%%%%%%%%%%%%%%%%%%%%
%%%%%%%%%%%%%%%%%%%%%%%%%%%%%%%%
\section{Proof of Theorem~\ref{MV Theorem 4}}\label{Proof of main theorem}
Let $q\in \NN_S$, $1\le r\le k$ and $D\in \mathfrak D_{r,q}$ be as in Notation~\ref{notations} (2). For $\Lambda \in \SG/\Gamma$,
define the set
\begin{equation}\label{Phi set}
\Phi(q,D)=\left\{\left(\ow_1, \ldots, \ow_r\right)\in (\ZZ_S^d)^r : \begin{array}{c}\rk(\ow_1, \ldots, \ow_r)=r,\\
\frac 1 q(\ow_1, \ldots, \ow_r)D\in (\ZZ_S^d)^k \end{array}\right\}
\end{equation}
and for $\sg\in \UL_d(\QQ_S)$, $\sg\Phi(q,D)=\{(\sg\ow_1, \ldots, \sg\ow_r) : (\ow_1, \ldots, \ow_r)\in \Phi(q,D)\}$.

Recall Notation \ref{notations}.
We first claim that for a continuous bounded function $f:(\QQ_S^d)^k \rightarrow \RR_{\ge 0}$ of compact support, we have
\begin{equation}\label{integral: partition}\begin{split}
\int_{\SG/\Gamma}\Siegel{k}{f}(\sg\Gamma)d\sg
&=\int_{\SG/\Gamma} \sum_{\ov_1, \ldots, \ov_k \in \sg\Gamma} f(\ov_1, \ldots, \ov_k) d\sg\\
&\hspace{-0.6in}=f(\origin,\ldots,\origin) 
+\int_{\SG/\Gamma}
%\sum_{\scriptsize \begin{array}{c}\ov_1, \ldots, \ov_k \in \sg\Gamma\\
%\rk\left(\ov_1, \ldots, \ov_k\right)=k \end{array}} 
\sum_{\scriptsize \begin{array}{c}
(\ow_1, \ldots, \ow_k)\\
\in \sg\Phi(1, \id_k)\end{array}}
f(\ov_1, \ldots, \ov_k)d\sg\\
&\hspace{-0.6in}+\sum_{r=1}^{k-1} \sum_{q\in \NN_S} \sum_{D \in \mathfrak D_{r,q}}
\int_{\SG/\Gamma}
\sum_{\scriptsize \begin{array}{c}(\ow_1, \ldots, \ow_r)\\ \in \sg\Phi(q,D)\end{array}}
f\Big(\frac 1 q \left(\ow_1, \ldots, \ow_r\right)D\Big)d\sg.
\end{split}\end{equation}

%%%%%%%%%%
For $\ov_1, \ldots, \ov_k\in \ZZ_S^d$, let $r=\rk\left(\ov_1, \ldots, \ov_k\right)$. 
Suppose that $r\neq 0, k$.
Then there are indices $1\le j_1< \cdots<j_r\le k$ such that 
\begin{enumerate}[1)]
\item $\rk\left(\ov_{j_1}, \ldots, \ov_{j_r}\right)=r$;
\item for $j_i\le j<j_{i+1}$,
$\ov_j \in \QQ \ov_{j_1}\oplus \cdots \oplus \QQ\ov_{j_{i}}$.
\end{enumerate}

%%%%%%%%%%
If we put $\ow_1=\ov_{j_1}, \ldots, \ow_r=\ov_{j_r}$,
there is $C\in \mathfrak M_{r,k}(\QQ)$ such that
$\left(\ow_1, \ldots, \ow_r\right)C=\left(\ov_1, \ldots, \ov_k\right)$ and
\begin{enumerate}[1)]
\item $[C]^{j_{i}}=\ve_{i}$, where $\{\ve_1, \ldots, \ve_r\}$ is a canonical basis of $\QQ_S^r$;
\item $C_{ij}=0$ for $1\le i\le r$ and $1\le j <j_i$.
\end{enumerate}

%%%%%%%%%%
Choose $q \in \NN_S$ such that $D=qC \in \mathfrak M_{r,k}(\ZZ_S)$ is reduced and
\begin{enumerate}[1)]
\item $[D]^{j_i}=q\ve_{i}$ and
\item $D_{ij}=0$ for $1\le i\le r$ and $1\le j <j_i$.
\end{enumerate}

Note that when $r=k$, only possible case is that $q=1$ and $D\in \mathfrak D_{r,q}$ is the identity matrix $\id_k$. Hence we have the following partition 
\[\begin{split}
&\left\{\left(\ov_1, \ldots, \ov_k\right) : \ov_j \in \ZZ_S^d\right\}\\
&\hspace{0.4in}= \{(\origin,\ldots,\origin)\}
\sqcup\left\{(\ov_1, \ldots, \ov_k) : \rk\left(\ov_1, \ldots, \ov_k\right)=k\right\}\\
&\hspace{0.7in}\sqcup\bigcupdot_{r=1}^{k-1}\bigcupdot_{q\in \NN_S}\bigcupdot_{D\in \mathfrak D_{r,q}}
\left\{\frac 1 q \left(\ow_1, \ldots, \ow_r\right)D : \left(\ow_1, \ldots, \ow_r\right) \in \Phi(q,D)\right\},
\end{split}\]
which shows the claim. 

Hence with \eqref{integral: partition}, Theorem~\ref{MV Theorem 4} is a direct consequence of the theorem below,
including the finiteness of the integral of $\Siegel{k}{f}$. 

%%%%%%%%%%%%%%%
\begin{theorem}\label{middle main}
Let $d\ge 3$ and $1\le k\le d-1$. 
Let $q\in \NN_S$, $1\le r\le k$ and $D\in \mathfrak D_{r,q}$ be as in Notation~\ref{notations} (2).
For a continuous bounded function $f:(\QQ_S^d)^k\rightarrow \RR_{\ge0}$ of compact support, it follows that
\[\begin{split}
&\int_{\SG/\Gamma} \sum_{\scriptsize \begin{array}{c}
(\ow_1, \ldots, \ow_r)\\
\in \sg\Phi(q,D)\end{array}}
f\Big(\frac 1 q \left(\ow_1, \ldots, \ow_r\right)D\Big)d\sg\\
&\hspace{1in}=\frac {N(D,q)^d} {q^{dr}}  \int_{(\QQ_S^d)^r} 
f\Big(\frac 1 q\left(\ov_1, \ldots, \ov_r\right)D\Big)d\ov_1 \cdots d\ov_r,
\end{split}\]
where $N(D,q)$ is defined as in Notation~\ref{notations} (3).
\end{theorem}

%%%
Before proving Theorem~\ref{middle main}, we need following propositions and the lemma.
%%%%%%%%%%%%%%%%%%%
\begin{proposition}\label{bounded of alpha}
Let $\alpha^{}_S$ be the function on the space of $S$-lattices in $\QQ_S^d$ defined by
\begin{equation}\label{alpha function}
\alpha^{}_S(\Lambda):=\sup_{1\le j\le d}\left\{\prod_{p\in S}\|\ov_1\wedge \cdots \wedge\ov_j\|_p^{-1} : \begin{array}{cc}
\ov_1, \ldots, \ov_j\in \Lambda,\\
\text{linearly independent}\end{array}\right\}.
\end{equation}
One can think of $\alpha^{}_S$ as the function on $\SG/\Gamma$.
Then for any $1\le r' <d$,
\[\int_{\SG/\Gamma} \alpha^{}_S(\sg\ZZ_S^d)^{r'} d\sg<\infty.\] 
\end{proposition}
\begin{proof} See Lemma 3.10 in \cite{HLM} for $\SL_d(\QQ_S)/\SL_d(\ZZ_S)$ and also Lemma 3.10 in \cite{EMM} for $\SL_d(\RR)/\SL_d(\ZZ)$.

Note that $\UL_d(\QQ_S)/\UL_d(\ZZ_S)\simeq \SL_d(\RR)/\SL_d(\ZZ)\times \prod_{p\in S_f} \UL_d(\ZZ_p)$.
Let $\pi:\UL_d(\QQ_S)/\UL_d(\ZZ_S) \rightarrow \SL_d(\RR)/\SL_d(\ZZ)$ be the projection defined by
\[\pi(\sg \Gamma)=g'_\infty \SL_d(\ZZ),\] 
where $g'_\infty\in \SL_d(\RR)$ is taken such that there is $(g'_p)_{p\in S_f}\in \prod_{p\in S_f}\UL_d(\ZZ_p)$ for which $\sg \Gamma=(g'_p)_{p\in S} \Gamma$.
Similar to the $\SL_d(\QQ_S)$-case, one can check that $\alpha^{}_S(\sg \ZZ_S^d)=\alpha(g'_\infty \ZZ^d)$, where
and $\alpha=\alpha_{\{\infty\}}$. 

Hence by replacing $\SL_d(\QQ_S)$ and $\SL_d(\ZZ_S)$ by $\UL_d(\QQ_S)$ and $\UL_d(\ZZ_S)$, respectively, in the proof of Lemma 3.10 in \cite{HLM}, we obtain the proposition.
\end{proof}

%%%%%%%%%%%%%%%%%%%
\begin{proposition}\label{Schmidt}
Suppose that a bounded function $f:\QQ_S^d\rightarrow \RR_{\ge 0}$ is supported on $B_{b_\infty}(\origin)\times \prod_{p\in S_f} p^{-b_p}\ZZ_p^d$, where $B_{b_\infty}(\origin)$ is the ball of radius $b_\infty$ centered at the origin.
There is $c_d>0$ depending only on the dimension $d$ such that for any unimodular lattice $\Lambda$, it holds that
\[
\widetilde f(\Lambda) < c_d \Big(b_\infty\times \prod_{p\in S_f} p^b_p\Big)^d \|f\|_{\sup} \cdot\alpha^{}_S(\Lambda).
\]
\end{proposition}
\begin{proof}
The real case is deduced from Lemma 2 in \cite{Sch1968}.
We may assume that $f$ is the indicator function of $B_{b_\infty}(\origin)\times \prod_{p\in S_f} p^{-b_p}\ZZ_p^d$ and $\Lambda=\sg\ZZ_S^d$, where $\sg=(g_p)_{p\in S}$ such that $g_p\in \UL_d(\ZZ_p)$ for $p\in S_f$. 

Since $g_p.(p^{-b_p}\ZZ_p^d)=p^{-b_p}\ZZ_p^d$, 
\[\begin{split}
\widetilde f(\sg\ZZ_S^d)
&=\#\Big(\left(\sg\ZZ_S^d-\{\origin\}\right) \cap \Big(B_{b_\infty}(\origin)\times \prod_{p\in S_f} p^{-b_p}\ZZ_p^d\Big)\Big)\\
&=\#\left(\left\{
\ov\in \ZZ_S^d : g_\infty\ov \in B_{b_{\infty}}(\origin)\;\text{and}\; \ov\in p^{-b_p}\ZZ_p^d 
\right\}\right)\\
&=\#\left(\left\{
\ow\in \prod_{p\in S_f} \ZZ^d : g_\infty\ow\in B_{b_\infty\times \prod p^{b_p}}(\origin)  
\right\}\right).
\end{split}\]
Here, we use the fact that $\ZZ_S^d \cap p^{-b_p}\ZZ_p^d=(1/p^{b_p})\ZZ^d$ and take $\ow$ by the projection of $\prod_{p\in S_f} p^{b_p}\:\ov$ to the infinite place.

Let $f'$ be the indicator function of $B_{b_\infty\times \prod_{p\in S_f} p^{b_p}}(\origin)$ in $\RR^d$. Since $\alpha^{}_{\{\infty\}}(g_\infty\ZZ^d)=\alpha^{}_S(\sg\ZZ_S^d)$, it follows that
\[\begin{split}
\widetilde f(\sg\ZZ_S^d)
&=\widetilde f'(g_\infty \ZZ^d)
<c_d \Big(b_\infty\times \prod_{p\in S_f} p^b_p\Big)^d \alpha^{}_{\{\infty\}}(g_\infty\ZZ^d)\\
&\hspace{1.5in}=c_d \Big(b_\infty\times \prod_{p\in S_f} p^b_p\Big)^d \alpha^{}_{S}(\sg\ZZ^d).
\end{split}\]
\end{proof}

%%%%%%%
\begin{lemma}\label{pixel lemma}
Let $q\in \NN_S$, $r\in \NN$ and $D\in \mathfrak D_{r,q}$ and $N(D,q)$ be as in Notation~\ref{notations} (2) and (3). Let $\Phi(q,D)$  be the set as in \eqref{Phi set}. Define
\[
\Psi(q,D)=\left\{(\ov_1, \ldots, \ov_r)\in (\ZZ_S^d)^r : \frac 1 q (\ov_1, \ldots, \ov_r)D \in (\ZZ_S^d)^k \right\}.
\]

Let $\pp=\prod_{p\in S_f} p$. For each $t\in \NN$, consider the set
\[
B_t=\Big(B_{\pp^t}(\origin)\times \hspace{-0.05in}\prod_{p\in S_f} p^{-t}\ZZ_p^d\Big)^r,
\]
where $B_{\pp^t}(\origin)$ is the ball of radius $\pp^t$ centered at the origin in $\RR^d$.
Then we have followings:
\begin{enumerate}[(a)]
%%%%%
\item
$\Psi(q,D)$ is an $S$-lattice in $(\QQ_S^d)^r$ whose covolume is
\[
\covol(\Psi(q,D))=q^{dr}/N(D,q)^d.
\]

%%%%%%
%\item For sufficiently large $t$,
%\[
%\#\left(\Psi(q,D)-\Phi(q,D)\right)
%\le r \Big(1+2\pp^{2t}\Big)^{d(r-1)}\times 2\Big((2\sqrt r)^t\pp^{2t}\Big)^{r-1}.
%\]
%In particular, $\#\left(\Psi(q,D)-\Phi(q,D)\right)/\vol(B_t)=o(\vol(B_t))$.

%%%%%
\item For any Borel set $B=\prod_{p\in S} B_p\subseteq (\QQ_S^d)^r$, where $B_\infty$ has a smooth boundary and $B_p\subseteq \QQ_p^d$ is such that $I_{B_p}$ is locally constant for $p<\infty$. 
For any compact set $\mathcal K\subseteq \SG/\Gamma$ and $\varepsilon>0$, there is $t_0\in \NN$ such that for any $\sg\in \mathcal K$ and $t\ge t_0$,
\[
\left|
\frac{\#(\pp^{-t}, p_1^t, \ldots, p_s^t)\sg\Psi(q,D) \cap B} {\mu(B)}
-\pp^{2tdr}\frac {N(D,q)^d}{q^{dr}}
\right|<\varepsilon.
\]
\end{enumerate}
\end{lemma}
\begin{proof}
\noindent (a)
Since $\Psi(q,D)$ is an additive subgroup and $(q\ZZ_S^d)^r \subseteq \Psi(q,D) \subseteq (\ZZ_S^d)^r$, $\Psi(q,D)$ is an $S$-lattice. Considering the projection $(\ZZ_S^d)^r\rightarrow (\ZZ_S^d/q\ZZ_S^d)^r$, it follows that the covolume of $\Psi(q,D)$ is
\[
\frac {\# (\ZZ_S^d/q\ZZ_S^d)^r}{\#\left\{(\ov_1, \ldots, \ov_r)\in (\ZZ_S^d/q\ZZ_S^d)^r : \frac 1 q (\ov_1, \ldots, \ov_r)D\in (\ZZ_S^d/q\ZZ_S^d)^k\right\}}=\frac {q^{dr}} {N(D,q)^d}
\]
by the definition of $N(D,q)$.

%\vspace{0.1in}
%%%%%%%%%%%%%%%%
%\noindent (b) Note that $(\ov_1, \ldots, \ov_r)\in \Psi(q,D)-\Phi(q,D)$ is linearly dependent.
% We may assume that $\ov_1$ is a linear combination of $\ov_2, \ldots, \ov_r$. 
%The number of possible choices of each $\ov_j$ ($2\le j\le r$) is at most $(1+2\pp^{2t})^d$.
%Since $\ov_1$ is an $S$-integral element lying in the intersection
%\[
%\QQ_S\text{-span of }\{\ov_2, \ldots, \ov_r\}
%\cap B_t,
%\]
%the number of possible $\ov_1$ is approximated by the $r'$-volume of the above region, where $r'$ is the rank of $\QQ_S$-span of $\{\ov_2, \ldots, \ov_r\}$, hence is bounded by $2\left((2\sqrt r)^tp^{2t}\right)^{r-1}$ for all sufficiently large $t$.

\vspace{0.1in}
%%%%%%%%%%%%%%
\noindent (b) We first consider the case when $S=\{\infty\}$ and $(\QQ_S^d)^r=(\RR^d)^r$.

Let $B_\infty\subseteq (\RR^d)^r$ be a bounded Borel set with piecewise smooth boundary for which $\mu_\infty(B_\infty)>0$, and let $\varepsilon>0$ be given. 
Let $\pp=\prod_{p\in S_f}p$.
It is well-known that for any lattice $\Lambda_\infty \subseteq (\RR^d)^r$, there is $t_0=t_0(B, \varepsilon, \Lambda_\infty)\in\NN$ such that 
\[
\left|\# \left(\pp^{-t}\Lambda_\infty \cap B_\infty\right)
-\pp^{tdr}\frac {\mu(B_\infty)}{\covol(\Lambda_\infty)} \right|<\varepsilon 
\]
for $t\ge t_0$. Note that (after fixing $B_\infty$ and $\varepsilon$,) $t_0$ can be taken associated with the diameter of the parallelopiped corresponding to $\Lambda_\infty$ in $(\RR^d)^r$ (see \cite{Sch1968} for example).

\vspace{0.1in}
%%%
Now consider a bounded Borel set $B=\prod_{p\in S} B_p\subseteq (\QQ_S^d)^r$, where $B_\infty$ has a smooth boundary and $B_p$ is locally constant for $p<\infty$. 

Recall that one can take a fundamental domain $\mathcal F$ of $\SG/\Gamma$ so that it is contained in $\mathcal F_\infty \times \prod_{p\in S_f}\UL_d(\QQ_p)$, where $\mathcal F_\infty$ is a fundamental domain for $\SL_d(\RR)/\SL_d(\ZZ)$.

Note that fundamental domains of $(\ZZ_S^d)^r$ and $(q\ZZ_S^d)^r$ are $([0,1)^d\times \prod_{p\in S_f} \ZZ_p^d)^r$ and $([0,q)^d \times \prod_{p\in S_f} \ZZ_p^d)^r$, respectively, since $q\in \NN_S$. 
Hence one can take a fundamental domain of $\Psi(q,D)$ by $\mathcal F'_\infty\times \prod_{p\in S_f} (\ZZ_p^d)^r$ for some Borel set $F'_\infty \subseteq ([0,q)^d)^r$ in $(\RR^d)^r$.

For any $S$-lattice of the form $\Lambda=\sg\Psi(q,D)$, where $\sg=(g_p)_{p\in S}\in \mathcal F$, 
since $g_p\ZZ_p^d=\ZZ_p^d$ for each $p\in S_f$,
we can take a fundamental domain for $(\QQ_S^d)^r/\Psi(q,D)$ as $g_\infty\mathcal F'_{\infty}\times \prod_{p\in S_f} (\ZZ_p^d)^r$. 

Hence for any given $\varepsilon>0$, there is $t_0=t_0(B, \varepsilon, \Lambda)\in \NN$ such that for $t\ge t_0$,
%%%
\begin{equation}\label{eq: pixel lemma}
\left|
\#\left((\pp^{-t}, p_1^t,\ldots, p_s^t)\Lambda \cap B \right)-\pp^{2tdr}\frac {\mu(B)}{q^{dr}/N(D,q)^d}\right|<\varepsilon
\end{equation}
and one can take $t_0\in \NN$ depending eventually only on the diameter of $\mathcal F'_\infty$.

%%%
Since the diameter of $\sg\mathcal F'_\infty$ is changed small in the small neighborhood in the space $\mathcal F.\Psi(q,D)$, for a given compact set $\mathcal K\subseteq \mathcal F$ and $\varepsilon>0$, one can take a uniform $t_0$ such that \eqref{eq: pixel lemma} holds for any $\Lambda \in \mathcal K.\Psi(q,D)$ and $t\ge t_0$. 
\end{proof}

%%%%%%%%%%%%%%%%%%%%%%%%%%%%%%%%%%%%
\begin{proof}[Proof of Theorem~\ref{middle main}]

For a continuous and bounded function $F:(\QQ_S^d)^r\rightarrow \RR_{\ge 0}$ of compact support, define the transform $\widehat F$ of $F$ by
\begin{equation}\label{widehat function}
\widehat F(\sg\Gamma)
:=\sum_{\scriptsize \begin{array}{c}
(\ow_1, \ldots, \ow_r)\\
\in \sg\Phi(q,D)\end{array}}
F(\ow_1, \ldots, \ow_r),\;\forall \sg\Gamma \in \SG/\Gamma.
\end{equation}

By putting
\begin{equation}\label{eq 2: middle main}
F(\ov_1, \ldots, \ov_r):=f\Big(\frac 1 q (\ov_1, \ldots, \ov_r)D\Big),
\end{equation}
it suffices to show that for any continuous bounded function $F:(\QQ_S^d)^r \rightarrow \RR_{\ge 0}$ of compact support,
\begin{equation}\label{eq 1: middle main}
\int_{\SG/\Gamma} \widehat F(\sg\Gamma) d\sg = \frac {N(D,q)^d} {q^{dr}}
\int_{\QQ_S^d} F(\ov_1, \ldots, \ov_r) d\ov_1\cdots d\ov_r.
\end{equation}

Set
\[\begin{split}
E&=\left\{(\ov_1, \ldots, \ov_r)\in (\QQ_S^d)^r:
\begin{array}{l}
\exists \:p\in S \text{ such that }\\[0.02in]
\hspace{0.2in}(\ov_1)_p, \ldots, (\ov_r)_p : \text{ linearly dependent}
\end{array}\right\}\\
&=\bigcup_{p\in S}
(\QQ_{S-\{p\}}^d)^r \times E_p,
\end{split}\]
where
\[
E_p=\bigcup_{i=1}^r\Big\{\Big((\ov_1)_p, \ldots, (\ov_{i-1})_p, 
\sum_{j\neq i} \beta_j(\ov_j)_p,
(\ov_{i+1})_p, \ldots, (\ov_r)_p\Big)\in (\QQ_p^d)^r:\beta_j \in \QQ_p\Big\}.
\]
%\[\begin{split}
%E_1&=\left\{(\ov_1, \ldots, \ov_r)\in (\QQ_S^d)^r: \ov_1, \ldots, \ov_r:\;\text{linearly dependent}\;\right\}\\
%&=\bigcup_{i=1}^r\left\{\Big(\ov_1, \ldots, \ov_{i-1}, 
%\sum_{j\neq i} \beta_j\ov_j,
%\ov_{i+1}, \ldots, \ov_r\Big):\beta_j \in \QQ_S\right\};\\
%E_2&=\text{the complement of }\Big(\prod_{p\in S}\left(\QQ_p^d-\{\origin\}\right)\Big)^r \text{ in }(\QQ_S^d)^r.
%\end{split}\]
%and $E=E_1 \cup E_2$.
Note that for each $p\in S$, $E_p$ is locally diffeomorphic to $\QQ_p^{d(r-1)+(r-1)}$ except the origin. Since we assume that $r\le k\le d-1$, $d(r-1)+(r-1)<dr$ so that $\mu_p(E_p)=0$. 
Here, $\mu_p=\mu_p^{dr}$ is the Haar measure of $(\QQ_p^d)^r\simeq \QQ_p^{dr}$ defined in Section~\ref{Statement of Results}. Since $\mu$ is the product measure of $\mu_p$, $\mu(E)=0$.
Moreover, it is easy to show that $\SG$ acts on $(\QQ_S^d)^r-E$ transitively.

\vspace{0.1in}
%%%%%%%%%%%%%%%%%%%%%%%%%%%%%%
\noindent\textbf{Step 1.}
We first show that $\psi:F\in C_c\left((\QQ_S^d)^r-E\right) \mapsto \int_{\SG/\Gamma} \widehat F d\sg$ is a $\SG$-invariant positive linear functional, which is obvious except $\int_{\SG/\Gamma} \widehat F d\sg<\infty$.

Since $F$ is bounded and compactly supported, there is a bounded function $f':\QQ_S^d\rightarrow \RR_{\ge 0}$ of compact support for which
$$F(\ov_1, \ldots, \ov_r)\le f'(\ov_1)\cdots f'(\ov_r)=:(f')^r(\ov_1, \ldots, \ov_r)$$
so that $\widehat F(\sg\Gamma) \le (\widetilde {f'}(\sg\Gamma))^r$. By Proposition~\ref{Schmidt}, there is a constant $c_{f'}>0$ such that $\widetilde {f'}<c_{f'}\alpha^{}_S$. Together with Proposition~\ref{bounded of alpha}, it holds that
\[
\int_{\SG/\Gamma} \widehat F d\sg
\le \int_{\SG/\Gamma} (\widetilde {f'})^r d\sg
\le  c_{f'}^r \int_{\SG/\Gamma}\alpha^{r}_S d\sg<\infty.
\]

By Rietz-Markov-Kakutani representation theorem, there is a unique Borel measure $\mu'$ for which $\int_{\SG/\Gamma} \widehat F d\sg=\int_{(\QQ_S^d)^r-E} F d\mu'$. 
Moreover, since $\psi$ is $\SG$-invariant, so is $\mu'$. Using the facts that $\SG$ acts on $(\QQ_S^d)^r-E$ transitively and $\mu(E)=0$, it follows that there is a constant $c>0$ such that
\[
\int_{\SG/\Gamma} \widehat F d\sg= c\int_{(\QQ_S^d)^r} F d\mu.
\]

\vspace{0.1in}
%%%%%%%%%%%%%%%%%%%%%%%%%%%%%%
\noindent\textbf{Step 2.} 
Let us now show that $c=N(D,q)^d/q^{dr}$.

It is known that for any $R>0$, 
\begin{equation}\label{eq 5: middle main}
\mathcal K_R:=\left\{\sg\Gamma\in \SG/\Gamma: \alpha^{}_S(\sg\Gamma)\le R\right\}
\end{equation}
is a compact set in $\SG/\Gamma$ and $\bigcup_{R>0} \mathcal K_R=\SG/\Gamma$. Moreover, since $\int_{\SG/\Gamma} \alpha^{}_S d\sg<\infty$, we have that $m(\mathcal K_R^c)\rightarrow 0$ as $R\rightarrow \infty$, where $\mathcal K_R^c=\SG/\Gamma -\mathcal K_R$:
\[\begin{split}
R\cdot m(\mathcal K_R^c)
&<R\int_{\mathcal K_R^c} 1 d\sg+\int_{\mathcal K_R} \alpha^{}_Sd\sg\\
&\le\int_{\mathcal K_R^c} \alpha^{}_S d\sg+\int_{\mathcal K_R} \alpha^{}_Sd\sg=\int_{\SG/\Gamma} \alpha^{}_S d\sg<\infty.
\end{split}\]

%%%
Let $\pp=\prod_{p\in S_f} p$. For each $t\in \NN$, consider the set
\[
B_t=\Big(B_{\pp^t}(\origin)\times \hspace{-0.05in}\prod_{p\in S_f} p^{-t}\ZZ_p^d\Big)^r - N_1(E),
\]
where $N_\delta(E)=\bigcup_{\ov\in E}(\ov+B_\delta)$, the $\delta$-neighborhood of $E$.
Take a sequence $\{F_t\}$ in $C_c\left((\QQ_S^d)^r-E\right)$ asymptotically approximating $\frac 1 {\mu(B_t)} I_{B_t}$ such that 
\[
0\le F_t \le \frac 1 {\mu(B_t)} I_{B_t},
\quad\text{and}\quad
F_t|_{B_{t-1}}=\frac 1 {\mu(B_t)} I_{B_t}|_{B_{t-1}}.
%\quad
%F_t|_{N_1(E)}\equiv 0.
\]
Then $\lim_{t\rightarrow \infty} \int_{(\QQ_S^d)^r} F_t d\mu=1$. 
%%%
Moreover, by Lemma~\ref{pixel lemma}, for any compact set $\mathcal K \subseteq \SG/\Gamma$ and $\varepsilon'>0$,
since $\Phi(q,D)\cap B_t=\Psi(q,D) \cap B_t$, there is $t_0=t_0(\varepsilon', \mathcal K)\in \NN$ so that for any $\Lambda\in \mathcal K$ and $t>t_0$,
\begin{equation}\label{eq 3: middle main}
\left|\widehat {F_t}(\Lambda) - \frac {N(D,q)^d} {q^{dr}} \right|<\varepsilon'.
\end{equation}
Here, we identify $\mathcal K$ with some compact set in a fundamental domain for $\SG/\Gamma$.

%%%%
By Proposition~\ref{Schmidt}, one can find $C>0$ such that for any $t\in \NN$,
\[
\int_{\SG/\Gamma} \widehat {F_t} d\sg
\le \int_{\SG/\Gamma} \sum_{\scriptsize \begin{array}{c}
(\ov_1, \ldots, \ov_r)\\
\in \sg(\ZZ_S^d)^r\end{array}}
\frac 1 {\mu(B_t)} I_{B_t} d\sg
<C\int_{\SG/\Gamma} \alpha^{r}_S d\sg.
\]

We want to show that
\begin{equation}\label{eq 4: middle main}
\lim_{t\rightarrow \infty}
\int_{\SG/\Gamma} \widehat F_t d\sg
%=\frac 1 {\covol(\Psi(q,D))} 
=\frac {N(D,q)^d} {q^{dr}}
=\lim_{t\rightarrow \infty} \frac {N(D,q)^d} {q^{dr}} \int_{(\QQ_S^d)^r} F_t d\mu.
\end{equation}

Let $\varepsilon>0$ be a given arbitrary small number.
For any $\mathcal K_R$,
\[\begin{split}
& \int_{\SG/\Gamma} \widehat {F_t} d\sg 
=\int_{\mathcal K_R} \widehat {F_t} d\sg+\int_{\mathcal K_R^c} \widehat {F_t} d\sg
\end{split}\]

By definition of $\mathcal K_R$ and Proposition~\ref{bounded of alpha}, since $r\le d-1$ and $\covol(\Psi(q,D))$ is at least $1$, one can find $R>0$ such that
\[\begin{split}
&\int_{\mathcal K_R^c} 
\left|\widehat {F_t}-\frac {N(D,q)^d} {q^{dr}}\right| d\sg
<C\int_{\mathcal K_R^c} \alpha^{r}_S d\sg + 
\frac {1-m(\mathcal K_R)} {q^{dr}/N(D,q)^d}\\
&\hspace{0.5in}\le \frac {C} {R^{1/2}} \int_{\mathcal K_R^c} \alpha^{r+1/2}_S d\sg+(1-m(\mathcal K_R))
<\frac {\varepsilon} 2.
\end{split}\]

Take $t_0=t_0(\mathcal K_R,\varepsilon/2)>0$ such that for $t>t_0$, \eqref{eq 3: middle main} holds for $\varepsilon'=\varepsilon/2$. Then it follows that
\[\begin{split}
\int_{\SG/\Gamma} \left|\widehat {F_t} - \frac {N(D,q)^d} {q^{dr}}\right| d\sg
&<\int_{\mathcal K_R} \left|\widehat {F_t} - \frac {N(D,q)^d} {q^{dr}}\right| d\sg+\int_{\mathcal K_R^c} \left|\widehat {F_t} - \frac {N(D,q)^d} {q^{dr}}\right| d\sg\\
&<\frac {\varepsilon} 2 + \frac {\varepsilon} 2 =\varepsilon.
\end{split}\]

Since $\varepsilon>0$ is an arbitrary small number, we obtain \eqref{eq 4: middle main}. 

\vspace{0.1in}
%%%%%%%%%%%%%%%%%%%%%%%%%%%%%%
\noindent\textbf{Step 3.}
Note that the function $F$ defined in \eqref{eq 2: middle main} is not a compactly supported function on $(\QQ_S^d)^r-E$.
Let us take a sequence $\{F'_t\}$ of functions in $C_c\left((\QQ_S^d)^r-E\right)$ for which
\[
0\le F'_t\le F
\quad\text{and}\quad
F'_t=F\;\text{on}\;(\QQ_S^d)^r-N_{-t}(E).
%\quad\text{and}\quad
%F'_t|_{N_{-t-1}}\equiv 0.
\]

Let $\mathcal K_R\subseteq \SG/\Gamma$ be as in \eqref{eq 5: middle main}. Then by Proposition~\ref{Schmidt}, since $F'_t$ is bounded by $F$, for any $t\in \NN$,
\[
\int_{\SG/\Gamma} \widehat {F} d\sg < C\int_{\SG/\Gamma} \alpha_S^r d\sg
\]
for some constant $C=C_F>0$.

Similar to the second step, from Proposition~\ref{bounded of alpha}, it follows that for a given $\varepsilon>0$, one can find $R>0$ such that for any $t\in \NN$,
\[ 
\int_{\mathcal K_R^c}
\widehat {F} d\sg < \frac C {R^{1/2}} \int_{\mathcal K_R^c} \alpha^{r+1/2} d\sg < \frac \varepsilon 3.
\]

Observe that since $\Phi(q,D)$ is an $S$-lattice not contained in $E$ and $F$ is compactly supported, for each $\sg\in \SG$, there is $t_0=t_0(\sg)\in \NN$ such that
\[
\sg\Phi(q,D)\cap N_{-t_0}(E)=\emptyset\;\text{on}\;\supp F,
\]
so that $\widehat {F'_t}(\sg\Lambda)=\widehat {F}(\sg\Lambda)$ for all $t>t_0$.
 
Using the compactness of $\supp F$ again, for each $\sg\Gamma$, one can find a small neighborhood $\mathcal N$ of $\sg\Gamma$ so that for $\sg'\Gamma\in \mathcal N$, $\sg'\Phi(q,D)\cap N_{-t_0-1}(E)=\emptyset$ on the support of $F$.
Hence one can take a uniform $t_0=t_0(\mathcal K_R)>0$ so that for any $\sg\Gamma \in \mathcal K_R$ and $t>t_0$, 
\[
\widehat {F'_t}(\sg\Gamma)=\widehat F(\sg\Gamma).
\]

Since $\int_{(\QQ_S^d)^r} F'_t d\mu\rightarrow \int_{(\QQ_S^d)^r} F d\mu$ as $t\rightarrow \infty$, it holds that for all sufficiently large $t\in \NN$, 
\[\begin{split}
&\left|\int_{\SG/\Gamma} \widehat F d\sg -\frac {N(D,q)^d} {q^{dr}}\int_{(\QQ_S^d)^r} F d\mu\right|\\
&\hspace{0.5in}\le\left|\int_{\SG/\Gamma} \widehat {F'_t}
 d\sg - \frac {N(D,q)^d}{q^{dr}} \int_{(\QQ_S^d)^r} F'_t d\mu \right|
+2\left|\int_{\mathcal K_R^c} \widehat F d\sg\right|\\
&\hspace{1.5in}+\frac {N(D,q)^d}{q^{dr}}\left|\int_{(\QQ_S^d)^r} F'_t d\mu - \int_{(\QQ_S^d)^r} F d\mu \right|<\varepsilon.
\end{split}\]
\end{proof}

%%%%%%%%%%%%%%%%%%%%%%%%%%%%%%%%%%%%%%%%%%%%%%%%%%%%%%%%%
\begin{proof}[Proof of Corollary \ref{Corollary L2 norm}]
By Theorem~\ref{MV Theorem 4} and Remark~\ref{sum over nonzero vectors},
\[\begin{split}
\int_{\SG/\Gamma} {\widetilde f\;}^2 d\sg
=\left(\int_{(\QQ_S^d)^2} f(\ov)d\ov\right)^2
\hspace{-0.05in}+\sum_{q\in \NN_S}\sum_{D\in \mathfrak D'_{1,q}}
\frac {N(D,q)^d} {q^d} \int_{\QQ_S^d} f'\times f'\big(\frac 1 q \ov D) d\ov.
\end{split}\]

The result follows from the fact that \textcolor{black}{the set of subscripts of sums} in the above equation is assorted into
%\begin{center}
%$\left\{\frac 1 q D : q\in \NN_S,\; D\in \mathfrak D_{1,q} \right\}
%=\left\{(1, 0)\right\}\cup \left\{(0,1)\right\}\cup\left\{(1,z) : 0\neq z\in \QQ\right\}.$%-\{0\}\right\}.$
%\end{center}
%
%In the last case, denote $z=w/(q\pp)$, where $q\in \NN_S$, $\pp\in \PP_S$ and $w\in \ZZ$ with $\gcd(q\pp, w)=1$.
\[
\left\{\frac 1 q D : q\in \NN_S,\; D\in \mathfrak D'_{1,q} \right\}
=\left\{(q,\frac w \pp) : q\in \NN_S,\; w\in \ZZ-\{0\},\; \gcd(q\pp, w)=1\right\}.\]%-\{0\}\right\}.$
\end{proof}
%%%%%%%%%%%%%%%%%%%%%%%%%%%%%%%%%%%%%%%%%%%%%%%%%%%%%%%%%

% Counterexample S-arithmetic %%%%%%%%%%%%%%%%%%%%%%%%%%%%%
\section{Absence of the uniform bound of the variance of the Siegel transform}

The following inequality is one of the consequences of Rogers' moment theorems together with the comparison of higher moments between the Siegel transform of a function $f$ and that of its \emph{spherical symmetrization} (see \cite{Rogers56-2} for the definition and the statement (Theorem 1) and also \cite{AM09, AM18} for details):
there is a constant $C_d>0$, depending only on the dimension $d$, such that for any bounded measurable set $A \subseteq \RR^d$, it holds that
\[
\int_{\SL_d(\RR)/\SL_d(\ZZ)} \widetilde{I_A}^2 dg \le \mu(A)^2 +C_d\:\mu(A),
\] 
or equivalently, $\int \left(\widetilde{I_A}-\mu(A)\right)^2 dg \le C_d\:\mu(A)$.
%Here, in the case when $f=I_{A}$ for some bounded set $A\subseteq \RR^d$, the indicator function of the ball whose volume is equal to that of $A$. 
%The above inequality derives the property which is significantly used in the application to problems of geometry of numbers (see Theorem \ref{LLUF Theorem 2.2} and Section \ref{Application}).

However, when $S\neq \{\infty\}$, such a uniform constant $C_d$ does not exist.
More precisely, one can show the following.

%%%%%%%%%%%%%%%%%%%
\begin{proposition}\label{Proposition 5.1}
Put $S=\{\infty, p\}$ for some odd prime $p$. There is a sequence $\{A_k\}_{k\in \NN}$ of positive measurable sets $A_k \subseteq \QQ_S^d$, $d\ge 3$, such that $\mu(A_k)\rightarrow \infty$ as $k\rightarrow \infty$ and there is a constant $c>0$ such that
\[
\int_{\SG/\Gamma} \widetilde{I_{A_k}}^2 d\sg - \mu(A_k)^2 \ge c\mu(A_k)^2.
\]
\end{proposition}
%%%%%%%%%%%%%%%%%
\begin{proof}
For each $k\in \NN$, define $A_k \subseteq \QQ_S^d$ by
\[\begin{split}
A_k&=\left\{
(\ov_\infty, \ov_p) \in \RR^d \times \QQ_p^d : \begin{array}{c}
\|\ov_\infty\|_\infty \|\ov_p \|_p \le 1, \\[0.02in]
\|\ov_\infty\|_\infty \le p^k\;\text{and}\; \|\ov_p\|_p \le p^k
\end{array}
\right\}\\
&=\bigsqcup_{m=0}^{2k-1}\left\{ \ov_p \in \QQ_p^d : \|\ov_p\|_p = p^{k-m} \right\}\times \left\{\ov_\infty \in \RR^d : \|\ov_\infty\|_\infty \le p^{-(k-m)} \right\}\\
&\hspace{0.1in}\sqcup \left\{\ov_p \in \QQ_p^d : \|\ov_p\|_p \le p^{-k} \right\}\times \left\{\ov_\infty : \|\ov_\infty\|_\infty \le p^{k}\right\}.
\end{split}\]

Since $\mu_p\left(\{\ov_p:\|\ov_p\|_p=p^j\}\right)=\mu_p\left(p^{-j}\ZZ_p^d- p^{-j+1}\ZZ_p^d\right)=p^{jd}-p^{(j-1)d}$, the volume of each of the first $2k$ sets in the disjoint union is
\[\begin{split}
&\mu\left(\left\{ \ov_p \in \QQ_p^d : \|\ov_p\|_p = p^{k-m} \right\}\times \left\{\ov_\infty \in \RR^d : \|\ov_\infty\|_\infty \le p^{-(k-m)} \right\}\right)\\
&\hspace{0.4in}=\left(1-\frac 1 {p^d}\right)p^{d(k-m)}\times B_d \frac 1 {p^{d(k-m)}}= B_d\left(1-\frac 1 {p^d}\right),
\end{split}\]
where $B_d$ is the volume of the unit ball of $\RR^d$. And the volume of the last set is
$$\mu\left(\left\{\ov_p \in \QQ_p^d : \|\ov_p\|_p \le p^{-k} \right\}\times \left\{\ov_\infty : \|\ov_\infty\|_\infty \le p^{k}\right\}\right)%=p^{-kd}\times B_d \:p^{kd}
=B_d.$$

Hence the volume $\mu(A_k)$ is
\[\begin{split}
\mu(A_k)
%&=\sum_{m=0}^{2k-1} \left(1-\frac 1 {p^d}\right)p^{d(k-m)}\times B_d \frac 1 {p^{d(k-m)}}+ p^{-dk}\times B_d \frac 1 {p^{-dk}}\\
&=2k\times B_d\left(1-\frac 1 {p^d}\right) + B_d,
\end{split}\]
and we have
\begin{equation}\label{prop 5.1 eq 1}
(2k+1)\left(1-\frac 1 {p^d}\right)B_d \le \mu(A_k) \le (2k+1)B_d.
\end{equation}

For $A\subseteq \QQ_S^d$ and $q \in \QQ$, denote by $A(q)=\{\ov \in \QQ_S^d : q\ov \in A\}$.
% the support of the function $I_A(q\:\cdot\:):\QQ_S^d \rightarrow \QQ_S^d$. 
By Corollary \ref{Corollary L2 norm},
\[
\int_{\SG/\Gamma} \widetilde{I_{A_k}}^2 d\sg - \mu(A_k)^2
=\sum_{q\in \NN_S}\sum_{\pp \in \PP_S} \hspace{-0.12in}\sum_{\scriptsize \begin{array}{c}w\in \ZZ-\{0\}\\ \gcd(q\pp, w)=1 \end{array}} \mu\left(A_k(q)\cap A_k\left(\frac w {\pp}\right)\right).
\]

For the lower bound, we count only the cases when $q\in \NN_S$, $\pp=p^m$, $0\le m \le 2k$ and $1\le w<q$. In this case, the intersection is
\[
A_k(q)\cap A_k\Big(\frac w {\pp}\Big)
=\left\{
(\ov_\infty, \ov_p) : \begin{array}{c}
\|\ov_\infty\|_\infty \|\ov_p\|_p \le \frac 1 q,\\
\|\ov_\infty\|_\infty \le \frac {p^k} q \text{ and } \|\ov_p\|_p\le p^{k-m}
\end{array}\right\},
\]
so that its volume is:
\[
\mu(A_k(q)\cap A_k\Big(\frac w {\pp}\Big))
=\frac {\mu(A_k)}{q^d}-m \times \left(1-\frac 1 {p^d}\right) \frac {B_d} {q^d}
\]
since $\mu(A(q))=\mu(A)/q^d$ for $q\in \NN_S$. Hence
\[\begin{split}
\int_{\SG/\Gamma} \widetilde{I_{A_k}}^2 d\sg - \mu(A_k)^2
%&=\sum_{q\in \NN_S}\sum_{\pp \in \PP_S} \hspace{-0.12in}\sum_{\scriptsize \begin{array}{c}w\in \ZZ-\{0\}\\ \gcd(q\pp, w)=1 \end{array}} \mu\left(A_k(q)\cap A_k\left(\frac w {\pp}\right)\right)\\
&\ge %\hspace{-0.12in}
\sum
%_{\scriptsize \begin{array}{c}
%q, w \in \NN_S\\
%\gcd(q,w)=1\\
%q>w \end{array}}\hspace{-0.12in}\sum_{m=0}^{2k} 
\left(\frac {\mu(A_k)}{q^d}-m\times\left(1- \frac 1 {p^d}\right)\frac {B_d} {q^d} \right)\\
&=
%\hspace{-0.12in}
\sum
%_{\scriptsize \begin{array}{c}
%q, w \in \NN_S\\
%\gcd(q,w)=1\\
%q>w \end{array}}\hspace{-0.12in} 
(2k+1)\frac {\mu(A_k)} {q^d} - \frac 1 2\cdot 2k (2k+1)\left(1- \frac 1 {p^d}\right)\frac {B_d} {q^d} \\
&\overset{\eqref{prop 5.1 eq 1}}{\ge} \sum
%_{\scriptsize \begin{array}{c}
%q, w \in \NN_S\\
%\gcd(q,w)=1\\
%q>w \end{array}}\hspace{-0.12in} 
(2k+1)\frac {\mu(A_k)} {q^d} - k \:\frac {\mu(A_k)} {q^d}\\
&\ge k\:\mu(A_k)
%\hspace{-0.12in}
\sum
%_{\scriptsize \begin{array}{c}
%q, w \in \NN_S\\
%\gcd(q,w)=1\\
%q>w \end{array}}\hspace{-0.12in} 
\frac 1 {q^d},
\end{split}\]
where the summation is over $q, w\in \NN_S$ for which $\gcd(q,w)=1$ and $q>w$ and $\pp=p^m$, $0\le m \le 2k$. Note that 
%where summation is over \textcolore{red}{$q, w\in \NN_S$, ...} 
%
%Since $d\ge 3$, 
\[
0< \hspace{-0.12in}\sum_{\scriptsize \begin{array}{c}
q, w \in \NN_S\\
\gcd(q,w)=1\\
q>w \end{array}}\hspace{-0.12in} \frac 1 {q^d}
\le \sum_{q\in \NN} \frac {\phi(q)}{q^d} = \frac {\zeta(d-1)}{\zeta(d)} <\infty,
\]
where $\phi(q)$ is the Euler totient function and $\zeta(d)$ is the Riemann-zeta function. 
Since $k$ is bounded below by some scalar multiple of $\mu(A_k)$ by  \eqref{prop 5.1 eq 1}, one can find a constant $c>0$ such that
\[\int_{\SG/\Gamma} \widetilde{I_{A_k}}^2 d\sg - \mu(A_k)^2 > c\mu(A_k)^2.
\]
\end{proof}
%%%%%%%%%%%

For the application in Section \ref{Application}, we will show that for certain families $\mathfrak F$ of measurable sets $A\subseteq \QQ_S^d$ with $\mu(A)\in (0,\infty)$, there is a constant $C_{\mathfrak F}>0$ for which the following holds:
\begin{equation}\label{class F}
\int_{\SG/\Gamma} \widetilde{I_A}^2 d\sg \le \mu(A)^2 + C_{\mathfrak F}\:\mu(A), \; \forall A\in \mathfrak F.
\end{equation}

Recall that for $A \subseteq \QQ_S^d$, we denote by $\T A$ the dilate of $A$ by $\T=(T_p)_{p\in S}$, which is given by
\[
\T A=\left\{(v_p)_{p\in S} \in \QQ_S^d : 
\Big(\frac {v_\infty} {T_\infty}, T_{p_1}v_{p_1}, \ldots, T_{p_s}v_{p_s} \Big) \in A \right\}.
\]
%%%%%%%%%%%%%%%%%%%%%%%%%%%%%%%%%%%%%%%
\begin{proposition}\label{Prop class F}
\begin{enumerate}[(a)]
\item For $A_0\subseteq \QQ_S^d$ with $0<\mu(A_0)<\infty$, let
\[
\mathfrak F_{A_0}=\Big\{ \T A_0 : \T \in \RR_{>0}\times \prod_{p\in S_f} p^{\ZZ} \Big\}.
\]
Then there is $C_{\mathfrak F_{A_0}}>0$ satisfying \eqref{class F}.
%%%%%
\item Define
\[
\mathfrak F_{prod}:=
\Big\{ \prod_{p \in S} A_p : A_p \subseteq \QQ_p^d \text{ with }\mu_p(A_p)<\infty, \; p\in S
\Big\}.
\]
Then there is $C_{\mathfrak F_{prod}}>0$ satisfying \eqref{class F}.
\item If there is a family $\mathfrak F$ with the property \eqref{class F},
then $$\mathfrak F \cup  \left\{\T_2 A - \T_1 A : A \in \mathfrak F,\; 0\preceq \T_1 \prec \T_2\right\}$$ also satisfies \eqref{class F} with the same constant $C_\mathfrak F>0$.
\end{enumerate}
\end{proposition}
%%%%%%%%%%%%%%%%%

We remark that if $S\neq \{\infty\}$, one cannot extend Propsition \ref{Prop class F} (c) as follow: if there is a family $\mathfrak F$ with the property \eqref{class F}, then $\mathfrak F \cup \{B - A : A \subseteq B \in \mathfrak F\}$ have the same property. Note that this property is one of key steps in the proof of Schmidt's theorem (\cite[Theorem 1]{Sch60}) so that one can generalize only the special case of Schmidt's theorem to the $S$-arithmetic case (Theorem \ref{AM Thm 1.3}).

%%%%%%%%%%%%%
\begin{proof}
\noindent (a) Choose any $C>0$ for which
\begin{equation}\label{prop 5.2 eq 1}
\int_{\SG/\Gamma} \widetilde{I_{A_0}}^2 d\sg \le \mu(A_0)^2 +C\:\mu(A_0).
\end{equation}

We claim that the proposition holds by taking $C_{\mathfrak F}=C$ for $\mathfrak F=\mathfrak F_{A_0}$.
According to Corollary \ref{Corollary L2 norm}, it suffices to show that 
\[
\sum_{q\in \NN_S}\sum_{\pp \in \PP_S} \hspace{-0.12in}\sum_{\scriptsize \begin{array}{c}w\in \ZZ-\{0\}\\ \gcd(q\pp, w)=1 \end{array}}\int_{\QQ_S^d} I_{\T A_0} (q\ov) I_{\T A_0} \Big(\frac w {\pp} \ov\Big) d\ov
\le C\mu(\T A_0).
\]

By the change of variables $\ow=(\ov_\infty/T_\infty, T_{p_1}\ov_{p_1}, \ldots, T_{p_s}\ov_{p_s})$ for $\T=(T_p)_{p\in S}$, 
\[\begin{split}
&\int_{\QQ_S^d} I_{\T A_0} (q\ov) I_{\T A_0} \Big(\frac w {\pp} \ov\Big) d\ov\\
&=\int_{\QQ_S^d} I_{\T A_0} (q(T_\infty, T^{-1}_{p_1},\ldots, T^{-1}_{p_s})\ow) I_{\T A_0} \Big(\frac w {\pp} (T_\infty, T^{-1}_{p_1},\ldots, T^{-1}_{p_s})\ow\Big) \:|\T|^d d\ow\\
&=|\T|^d \int_{\QQ_S^d} I_{A_0} (q\ow) I_{A_0}\Big(\frac w {\pp} \ow\Big) d\ow.
\end{split}\] 

Hence
\[\begin{split}
&\sum_{q\in \NN_S}\sum_{\pp \in \PP_S} \hspace{-0.12in}\sum_{\scriptsize \begin{array}{c}w\in \ZZ-\{0\}\\ \gcd(q\pp, w)=1 \end{array}}\int_{\QQ_S^d} I_{\T A_0} (q\ov) I_{\T A_0} \Big(\frac w {\pp} \ov\Big) d\ov\\
&\hspace{0.4in}=|\T|^d\sum_{q\in \NN_S}\sum_{\pp \in \PP_S} \hspace{-0.12in}\sum_{\scriptsize \begin{array}{c}w\in \ZZ-\{0\}\\ \gcd(q\pp, w)=1 \end{array}} \int_{\QQ_S^d} I_{A_0} (q\ow) I_{A_0}\Big(\frac w {\pp} \ow\Big) d\ow\\
&\hspace{0.4in}\le |\T|^d C\:\mu(A_0)= C \:\mu(\T A_0),
\end{split}\]
where the last inequality follows from \eqref{prop 5.2 eq 1} and Corollary \ref{Corollary L2 norm}.

\vspace{0.2in}
%%%%%
\noindent (b) Now, consider $A=\prod_{p\in S} A_p\in \mathfrak F_{prod}$, where $A_p \subseteq \QQ_p^d$, $p\in S$, has a positive measure.
Similar to (a), we need to find a uniform constant $C>0$ such that
\[
\sum_{q\in \NN_S}\sum_{\pp \in \PP_S} \hspace{-0.12in}\sum_{\scriptsize \begin{array}{c}w\in \ZZ-\{0\}\\ \gcd(q\pp, w)=1 \end{array}}\int_{\QQ_S^d} I_{A} (q\ov) I_{A} \Big(\frac w {\pp} \ov\Big) d\ov
\le C\:\mu(A).
\]

Note that $A(q)\cap A\Big(\frac w {\pp}\Big)= \prod_{p\in S_f} A_p(q) \cap A_p \Big(\frac w {\pp} \Big)$ and
\[\begin{split}
\mu_\infty(A_\infty(q)\cap A_\infty\Big(\frac w {\pp}\Big))&\le\left\{
\begin{array}{cl}
\frac 1 {q^d}\: \mu_\infty(A_\infty), & \text{if } q>\frac w {\pp};\\
\frac 1 {|w/\pp|^d}\: \mu_\infty(A_\infty), & \text{if } q< \frac w {\pp}.
\end{array}\right.\\[0.12in]
%%%%%
\mu_p(A_p(q)\cap A_p\Big(\frac w {\pp}\Big))&\le\left\{
\begin{array}{cl}
|\pp|^d_p\: \mu_p(A_p), & \text{if } p\:|\:\pp;\\
\mu_p(A_p), & \text{otherwise.}
\end{array}\right.
\end{split}\]

As in the proof of Proposition \ref{Proposition 5.1}, for $A\subseteq \QQ_S^d$ and $q\in \QQ$, set $A(q)=\left\{\ov\in \QQ_S^d : q\ov \in A\right\}$.
It follows that 
\[\begin{split}
\mu(A(q)\cap A\Big(\frac w {\pp}\Big))
&\le\frac 1 {\max(q, |w/\pp|)^d}\:\mu_\infty(A_\infty)\times 
\prod_{p\in S_f}|\pp|_p^d \:\mu_p(A_p)\\
&=\frac 1 {\max(q, |w/\pp|)^d}\:\mu_\infty(A_\infty)\times 
\pp^{-d}\prod_{p\in S_f}\mu_p(A_p)\\
&=\frac 1 {\max(q\pp, |w|)^d} \:\mu(A).
\end{split}\]

Hence
\[\begin{split}
&\sum_{q\in \NN_S}\sum_{\pp \in \PP_S} \hspace{-0.12in}\sum_{\scriptsize \begin{array}{c}w\in \ZZ-\{0\}\\ \gcd(q\pp, w)=1 \end{array}}\hspace{-0.15in}\int_{\QQ_S^d} I_{A} (q\ov) I_{A} \Big(\frac w {\pp} \ov\Big) d\ov
=\sum_{q,\pp,w} \mu(A(q)\cap A\Big(\frac w {\pp}\Big))\\
&\hspace{0.4in}\le\sum_{q,\pp,w}
\frac 1 {\max(q\pp, |w|)^d} \:\mu(A)
=\sum_{q' ,w}
\frac 1 {\max(q', |w|)^d} \:\mu(A),
\end{split}\]
where $q'=q\pp$. Since the quantities in the summation are symmetric with respect to the cases when [$q'>w$ and $q'<w$], and [$w$ and $-w$, $w\in \NN$], respectively,

\[\begin{split}
&\sum_{q'\in \NN} \hspace{-0.12in}\sum_{\scriptsize \begin{array}{c}w\in \ZZ-\{0\}\\ \gcd(q\pp, w)=1 \end{array}}\frac 1 {\max(q', |w|)^d} \:\mu(A)=4\times\hspace{-0.17in}\sum_{\scriptsize \begin{array}{c}
q',w \in \NN, 0<w<q' \\
\gcd(q',w)=1\end{array}}
\frac 1 {(q')^d} \:\mu(A)\\
&\hspace{0.4in}=4\times \sum_{q'\in \NN} \frac {\phi(q')}{(q')^d} \:\mu(A)
=\frac{4\zeta(d-1)}{\zeta(d)} \:\mu(A).
\end{split}\]

Therefore if we take $C_{\mathfrak F_{prod}}= 4\zeta(d-1)/\zeta(d)$, it follows that
for any $A \in \mathfrak F_{prod}$,
\[
\int_{\SG/\Gamma} \widetilde{I_A}^2 d\sg \le \mu(A)^2 + C_{\mathfrak F_{prod}}\:\mu(A).
\]

\vspace{0.2in}
%%%%%
\noindent (c) 
For $A\in \mathcal F$, let $C_A>0$ be a positive constant such that
\[
\int_{\SG/\Gamma} \widetilde{I_A}^2 d\sg=\mu(A)^2+C_A\:\mu(A).
\] 
Then the property \eqref{class F} implies that $C_A\le C_{\mathfrak F}$. Moreover, by the proof of (a), it follows that for any $\T\succeq 0$,
\[
\int_{\SG/\Gamma} \widetilde{I_{\T A}}^2 d\sg=\mu(\T A)^2+C_A\:\mu(\T A).
\]
Given $\T_1\prec \T_2$, let $B=\T_2 A- \T_1 A$.
For $q\in \NN_S$, $\pp \in \PP_S$ and $w\in \ZZ$, since
\[\begin{split}
&B(q)\cap B\Big(\frac w {\pp}\Big) \subseteq 
\T_2 A(q)\cap \T_2 A\Big(\frac w {\pp}\Big) - \T_1 A(q)\cap \T_1 A\Big(\frac w {\pp}\Big)\;\text{and} \\
&\T_1 A(q)\cap \T_1 A\Big(\frac w {\pp}\Big) \subseteq
\T_2 A(q)\cap \T_2 A\Big(\frac w {\pp}\Big),
\end{split}\]
it follows that
\[\begin{split}
&\sum_{q\in \NN_S}\sum_{\pp \in \PP_S} \hspace{-0.12in}\sum_{\scriptsize \begin{array}{c}w\in \ZZ-\{0\}\\ \gcd(q\pp, w)=1 \end{array}} \hspace{-0.12in}
\mu(B(q)\cap B\Big(\frac w {\pp}\Big))\\
&\hspace{0.4in}\le\sum_{q,\: \pp,\: w}
\mu(\T_2 A(q)\cap \T_2 A\Big(\frac w {\pp}\Big))
-\sum_{q,\: \pp,\: w}
\mu(\T_1 A(q)\cap \T_1 A\Big(\frac w {\pp}\Big))\\
&\hspace{0.4in}\le C_A\: \mu(\T_2 A) - C_A\: \mu(\T_1 A)
=C_A\: \mu(B) \le C_{\mathfrak F}\: \mu(B).
\end{split}\]

Hence the result follows from Corollary \ref{Corollary L2 norm}.
\end{proof}
%%%%%%%%%%%

%%%%%%%%%%%%%%%%%%%%%
%%%%%%%%%%%%%%%%%%%%%
\section{Application}\label{Application}
%%%%%%%%%%%%%%%%%%%%%
%%%%%%%%%%%%%%%%%%%%%
% Random Minkowski %%%%%%%%%%%%%%%%%%%%
The following theorem is a generalization of Theorem 2.2 in \cite{AM09}. The proof is not much different from that of \cite[Theorem 2.2]{AM09} but we include it for  completeness.

%%%%%%%%%%%%%%%%%%%%%%%%%%%%%%%%%%%%%%%
\begin{theorem}\label{LLUF Theorem 2.2} 
Let $\mathfrak F$ be a family of measurable sets $A$ with $\mu(A)\in (0,\infty)$ such that there is a constant $C_\mathfrak F>0$ for which
\[
\int_{\SG/\Gamma} \widetilde{I_A}^2 d\sg \le \mu(A)^2 +C_{\mathfrak F}\:\mu(A),\; \forall A\in \mathfrak F.
\]

Then for any $A\in \mathfrak F$, we have
%\begin{flalign*}
%\mathrm{(a)}&\;
\[
m\left(\left\{\sg\Gamma \in \SG/\Gamma : (\sg\ZZ_S^d-\{\origin\}) \cap A = \emptyset \right\}\right)\le \frac {C_{\mathfrak F}}{\mu(A)}.
\]
\end{theorem}
%%%%%%%%%%
\begin{proof}
Given $A \in \mathcal F$, define
\[\mathcal U:=\left\{\sg \Gamma : (\sg \ZZ_S^d-\{\origin\}) \cap A=\emptyset \right\}\subseteq \SG/\Gamma
\]
and denote $u=m(\mathcal U)$ and $a=\mu(A)$. Define functions on $\SG/\Gamma$ by 
\[\begin{split}
\phi&=\phi_{A}=\widetilde{I_{A}},\\
\psi&=\psi_{\mathcal U}=I_{\mathcal U^c},
\end{split}\]
where $I_A$ and $I_{\mathcal U^c}$ are the indicator functions of $A$ and the complement of $\mathcal U$ in $\QQ_S^d$ and $\SG/\Gamma$, respectively.
Since 
\[\psi(\sg \Gamma)=0 \Leftrightarrow \sg\ZZ_S^d \cap A = \emptyset
\Rightarrow \phi(\sg\Gamma)=0,
\]
we obtain the equality $\phi=\phi\cdot\psi$.
From Siegel's integral formula, it follows that
$$a=\|\phi\|_{\mathcal L^1}\quad\text{and}\quad1-u=\|\psi\|_{\mathcal L^1}=\|\psi\|_{\mathcal L^2}.$$ % since $\psi$ is an indicator function.

Then Cauchy-Schwarz inequality says that
\begin{equation}\label{THII Theorem 2.2 eq 1}
\begin{split}
a^2&=\left(\int_{\SG/\Gamma} \phi d\sg \right)^2 
\le \left(\int_{\SG/\Gamma} \phi^2 d\sg \right)\left(\int_{\SG/\Gamma} \psi^2 d\sg \right)\\
&=\left(\int_{\SG/\Gamma} \phi^2 d\sg \right)(1-u).
\end{split}\end{equation}

By assumption of $\mathfrak F$, since
%\begin{equation}\label{THII Theorem 2.2 eq 2}
$\int_{\SG/\Gamma} \phi^2 d\sg 
\le a^2 + C_{\mathfrak F}\: a$,
%\end{equation}
%Take $C_d=4\zeta(d-1)/\zeta(d)$. 
%From \eqref{THII Theorem 2.2 eq 1} and \eqref{THII Theorem 2.2 eq 2}, 
\[
u \le 1 - \frac {a^2} {a^2+C_{\mathfrak F}\: a} =\frac {C_{\mathfrak F}\: a} {a^2 + C_{\mathfrak F}\: a} \le \frac {C_{\mathfrak F}} {a}.
\]
\end{proof}
%%%%%%%%%%%

%%%%%%%%%%%%%%%%%%%%%%%%%%%%%%%%%%%%%%%%%%%%%%%%%%%%%%%%%%%%%%%%%%%%%%%%
\begin{proposition} \label{HLM Theorem 1.2}
Let $d\ge 3$.
%Let $\{\I_{\j}\subset \QQ_S : \j \in \NN^{s+1}\}$ be a family of bounded open sets such that $\I_{\j} \subseteq \I_{\j'}$ if $\j \succeq \j'$. 
Let $\q{}{}$ be a nondegenerate isotropic quadratic form on $\QQ_S^d$ and $R>0$ be any positive number. For any $\I \subseteq \QQ_S$ such that $\sup_{\ov\in \I} \|\ov\|_S <R$, we have
\[
\mu(\q{-1}{}(\I)\cap B(\origin, \T))= c_{\q{}{}}\; \mu(\I)|\T|^{d-2}+o_{\q{}{}, R}(|\T|^{d-2}),
\]
where $B(\origin, \T)=\{ \ov \in \QQ_S^d : \|\ov\|_p < T_p, \;\forall p \in S\}$. 
\end{proposition}
\begin{proof}
It suffices to show that each $p\in S$,
\[
\mu_p(q^{-1}_p(I_p)\cap B(\origin, T_p))
=c_{q_p}\mu_p(I_p)T_p^{d-2}+ o_{q_p}(|I_p|_pT_p^{d-2}),
\]
where $\q{}{}=(q_p)_{p\in S}$, $\I=(I_p)_{p\in S}$ and $\T=(T_p)_{p\in S}$. 
For $p=\infty$, see Theorem 5 in \cite{KY} and for $p\in S_f$, it follows from the proof of Proposition 4.2 in \cite{HLM}. 
%See Lemma 3.8 in \cite{EMM} for the case that $S=\{\infty\}$ and Proposition 1.2, Proposition 4.2 in \cite{HLM} for general $S$.
%Note that the collection $\{I_{\I_{\j}}\}$ of indicator functions of $\I_{\j}$'s is equicontinuous.
%%%%%%%%%
%Proposition 1.2 in \cite{HLM} does not mention the equicontinuous family of functions, but it is not hard to deduce the proposition from the proof of the proposition 4.2 in \cite{HLM}.  
\end{proof}

%%%%%%%%%%%%%%%%%%%%%%%%%%%%%%%%%%%%%%%%%%%%%%%%
\begin{proof}[Proof of Theorem~\ref{AM Thm 1.1}] 
Fix an isotropic quadratic form $\q{0}{}$ on $\QQ_S^d$ and a constant $\xi= (\xi_p)_{p\in S}\in \QQ_S$. 
%(\xi_\infty, \xi_{1}, \ldots, \xi_{s})

For any $\j=(j_p)_{p\in S} \in \NN^{s+1}$,%(j_{\infty}, j_{1}, \ldots, j_{s})\in \NN^{s+1}$,
define
\[\begin{split}
% A_\j does not contain 0 since q^0_\infty(v)\neq 0.
\mathcal A_\j&=\left\{\ov \in \QQ_S^d : 
\begin{array}{c}
\left|q^0_\infty(\ov)-\xi_\infty \right|_\infty< e^{-j_\infty};\\
\left|q^0_{p}(\ov)-\xi_p \right|_p < p^{-j_p},\;p\in S_f\\
\end{array}\right\} \\
&\hspace{1.in}\cap
\left\{\ov \in \QQ_S^d :
\begin{array}{c}
\|\ov\|_\infty<e^{(\frac 1 {d-2}+\delta)(j_\infty-1)};\\
\|\ov\|_{p}<{p}^{(\frac 1 {d-2}+\delta)j_p},\;p\in S_f\\
\end{array}\right\}.
\end{split}\]

Note that $\mathcal A_\j$ is contained in $\mathfrak F_{prod}$, which is defined in Proposition \ref{Prop class F}.
By Proposition~\ref{HLM Theorem 1.2}, there is $\j_0\succ 0$ such that for any $\j\succ \j_0$,
\[\begin{split}
\mu\left(\mathcal A_\j\right)
&> c_{\q{}{0}} \Big(e^{-j_\infty}\cdot\prod_{p \in S_f} p^{-(j_p+1)}\Big)\times
\Big(e^{(\frac 1 {d-2} + \delta)(j_\infty-1)}\cdot\prod_{p\in S_f} p^{(\frac 1 {d-2} + \delta)(j_p-1)}\Big)^{d-2}\\
&= c_{\q{}{0}}(ep_1\cdots p_s)^{-1} \left(e^{j_\infty-1}p_1^{j_{p_1}-1}\cdots p_s^{j_{p_s}-1}\right)^{(d-2)\delta}\\
&=c'_0  \|\exp(\j-\one)\|^{(d-2)\delta},
\end{split}\]
where $c'_0=c_{\q{}{0}}e^{-1}(p_1\cdots p_s)^{-2}$ and $\j-\one:=(j_\infty-1, j_{p_1}-1, \ldots, j_{p_s}-1)$.

By Theorem \ref{LLUF Theorem 2.2},
\[m\left(\left\{\sg \Gamma \in \SG/\Gamma : (\sg \ZZ_S^d-\{\origin\}) \cap \mathcal A_\j=\emptyset \right\}\right)
<\frac {C_{\mathfrak F_{prod}}} {c'_0} \|\exp(\j-\one)\|^{-(d-2)\delta}. 
\]

Hence
\begin{align*}
&m\left(\left\{\sg\Gamma \in \SG/\Gamma : (\sg \ZZ_S^d-\{\origin\})\cap \mathcal A_\j=\emptyset\;\text{for infinitely many}\;\j\right\}\right)%\label{null set}\tag{*}
\\
&\le m\left(\bigcap_{\j_0\succeq 0}\;\bigcup_{\j\succeq \j_0}\left\{\sg\Gamma : (\sg \ZZ_S^d-\{\origin\}) \cap \mathcal A_\j=\emptyset\right\}\right)\\
&\le\lim_{\j_0\rightarrow \infty} \frac {C_{\mathfrak F_{prod}}}{c'_0} \sum_{\j\succeq \j_0}\|\exp(\j-\one)\|^{-(d-2)\delta}=0. 
\end{align*}

Consider an isotropic quadratic form $\q{}{}$ such that 
$$\q{}{}(\ov)={\q{0}{}}(\sg \ov),\;\forall \ov \in \QQ_S^d$$
for some $\sg \in \SG$ and suppose that $\sg \ZZ_S^d-\{\origin\}$ 
intersects with all but finitely many $\mathcal A_\j$'s.
By the Borel-Cantelli lemma, the set of such quadratic forms $\q{}{}$ has full measure.

Choose $\varepsilon_0>0$ such that 
$(\sg \ZZ_S^d-\{\origin\}) \cap \mathcal A_\j \neq \emptyset$ 
whenever $e^{-j_\infty}<\varepsilon_0$ and $p^{-j_p}<\varepsilon_0$, $\forall p\in S_f$.

For any positive $\varepsilon < \varepsilon_0$, 
there is $\j=(j_\infty, j_{p_1}, \ldots, j_{p_s})$ such that 
$$e^{-j_\infty}< \varepsilon \le e^{-j_\infty+1}\quad\text{and}\quad 
p^{-j_p}<\varepsilon\le p^{-j_p+1},\;p\in S_f.$$ 

Since $(\sg \ZZ_S^d-\{\origin\}) \cap \mathcal A_\j\neq \emptyset$, there is a nonzero $\ow=\sg\ov \in \sg \ZZ_S^d$ so that
\[\begin{split}
\left|q_\infty(\ov)-\xi_\infty\right|&=\left|q^0_\infty(\ow)-\xi_\infty\right|_\infty< e^{-j_\infty}<\varepsilon,\\
\left|q_p(\ov)-\xi_p\right|&=\left|q^0_p(\ow)-\xi_p\right|_{p}< p^{-j_p}<\varepsilon,\; p\in S_f,
\end{split}\]
and
\[\begin{split}
\|\ov\|_\infty&=\|\sg^{-1}(\sg\ov)\|_\infty 
\le \|\sg^{-1}\|_{op}\|\sg \ov\|_\infty
\le \|\sg^{-1}\|_{op}\left(e^{j_\infty -1}\right)^{\frac 1 {d-2}+\delta}\\
&\le c_{\q{}{}} \varepsilon^{-\left(\frac 1 {d-2}+\delta\right)},\\
\|\ov\|_{p}&=\|\sg^{-1}(\sg\ov)\|_{p} 
\le \|\sg^{-1}\|_{op}\|\sg \ov\|_{p}
\le \|\sg^{-1}\|_{op}\left(p^{j_p -1}\right)^{\frac 1 {d-2}+\delta}\\
&\le c_{\q{}{}} \varepsilon^{-\left(\frac 1 {d-2}+\delta\right)},\; p\in S_f,\\
\end{split}\]
where $c_{\q{}{}}$ is the operator norm $\|\sg^{-1}\|_{op}=\max_{p\in S} \{\|g_p^{-1}\|_{op} \}$ of $\sg^{-1}=(g_p^{-1})_{p\in S}$.%:\QQ_S^d\rightarrow \QQ_S^d$.
\end{proof}

%%%%%%%%%%%%%%%%%%%%%%%%%%%%%%%%%%
\begin{theorem}\label{Sch Lemma 2} 
Let $\mathfrak F$ be the family of star-shaped sets centered at the origin for which the conclusion of Theorem \ref{LLUF Theorem 2.2} holds.
Let $A\subseteq \mathfrak F$ and $N_1$, $N_2\in \NN$ with $N_2 \ge N_1$. For each $\CJ^f\in \prod_{p\in S_f}p^{\NN}$, 
define
\[
R_{A} (N_1, N_2, \CJ^f) (\sg \ZZ_S^d)
:= \widetilde I_{(j_2, \CJ^f)A -(j_1, \CJ^f)A} (\sg \ZZ_S^d) - (N_2-N_1)\mu(A),
\]
where $j_1, j_2\in \RR_{>0}$ are chosen so that $|(j_i, \CJ^f)|^d=N_i$, $i=1,2$. For each $\ell\in \NN$, put
\[
K_{\ell}
:=\left\{(u2^t, (u+1)2^t) : \begin{array}{c}
t=0,1,2, \ldots, \ell \\
u=0,1,2, \ldots, 2^{\ell-t}-1\\
\end{array}
\right\}.
\]
For any positive function $\psi$ on $\NN$, consider the set
\[
\mathcal B_{\ell}
:=\left\{
\sg\Gamma \in \SG/\Gamma : 
\sum_{(N_1, N_2)\in K_{\ell}} R_{A}^2(N_1, N_2, \CJ^f)(\sg \ZZ_S^d) \ge (\ell+1)2^{\ell}\psi(\ell)  
\right\}.
\]
Then the set $\mathcal B_{\ell}$ is independent of the choice of $\CJ^f$ and  
there is a constant $C(A,d)>0$ such that
\[
m(\mathcal B_{\ell})< C(A, d)\psi^{-1} (\ell).
\]
\end{theorem}
%%%%%%%%%%
\begin{proof}
For any $\CJ^f$, ${\CJ^f}'\in \prod_{p\in S_f} p^{\NN}$, let
\[
N_1=|(j_1,\CJ^f)|^d=|(j'_1, {\CJ^f}')|^d\quad\text{and}\quad
N_2=|(j_2,\CJ^f)|^d=|(j'_2, {\CJ^f}')|^d.
\] 
Then there is $\pp \in \PP_S$ such that $(j_i, \CJ^f)=\pp(j'_i, {\CJ^f}')$, $i=1,2$. Hence
\[\begin{split}
R_{A}(N_1, N_2, \CJ^f)
&= \widetilde I_{\pp(j'_2, {\CJ^f}')A -\pp(j'_1, {\CJ^f}')A} (\sg \ZZ_S^d) - (N_2-N_1)\mu(A)\\
&= \widetilde I_{(j'_2, {\CJ^f}')A -(j'_1, {\CJ^f}')A} (\sg (\pp^{-1}\ZZ_S^d)) - (N_2-N_1)\mu(A)\\
&=R_{A}(N_1, N_2, {\CJ^f_2}'),
\end{split}\]
since $\pp^{-1}\ZZ_S^d=\ZZ_S^d$.
Therefore the choice of $\CJ^f$ is not relavant to the definition of $\mathcal B_{\ell}$. 

\vspace{0.1in}
%%%%%%%%%%
Now, denote $B=(j_2,\CJ^f)A-(j_1,\CJ^f)A$ so that $\mu(B)=(N_2-N_1)\mu(A)$. 

By Proposition \ref{Prop class F} (c), it follows that
\[\begin{split}
\int_{\SG/\Gamma} R^2_A(N_1,N_2,\CJ^f) (\sg \ZZ_S^d) d\sg
&=\int_{\SG/\Gamma} \widetilde{I_B}^2- 2\mu(B)\widetilde{I_B}+\mu(B)^2 d\sg\\
&\hspace{-0.5in}=\int_{\SG/\Gamma} \widetilde I_{B} d\sg - \mu(B)^2 \le C_{\mathfrak F}\: \mu (B)=C_{\mathfrak F}\:\mu(A)(N_2-N_1).
\end{split}\]

Put $C(A,d)=C_{\mathfrak F}\:\mu(A)$. Then
\[\begin{split}
&\int_{\SG/\Gamma} \sum_{(N_1, N_2)\in K_{\ell}} R_{A}^2(N_1, N_2, \CJ^f)(\sg \ZZ_S^d) d\sg \\
&\hspace{0.7in}=\sum_{t=0}^{\ell}\sum_{u=0}^{2^{\ell-t}-1}
\int_{\SG/\Gamma} R^2_A(u2^t, (u+1)2^t, \CJ^f)(\sg \ZZ_S^d) d\sg\\
&\hspace{0.7in}\le \sum_{t=0}^{\ell} C(A,d) 2^{\ell} = C(A,d) (\ell+1)2^{\ell}.
\end{split}\]

Hence
\[\begin{split}
m(\mathcal B_{\ell})
&\le \int_{\mathcal B_{\ell}}
\frac {\sum_{(N_1, N_2)\in K_{\ell}} R_{A}^2(N_1, N_2, \CJ^f)(\sg \ZZ_S^d)}
{(\ell+1)2^{\ell}\psi(\ell)} d\sg\\
&\le \frac {C(A,d)(\ell+1)2^{\ell}}{(\ell+1)2^{\ell} \psi(\ell)}
\le C(A,d) \psi^{-1}(\ell).
\end{split}\]
\end{proof}

%%%%%%%%%%%%%%%%%%%%%%%%%%%%%%%%%%%%%%%%%%%%%%%%%
\begin{proof}[Proof of Theorem~\ref{AM Thm 1.3}]
Define
\begin{center}
$\Psi:=\{(j, \CJ^f)\in \RR_{\ge 1}\times \prod_{p\in S_f} p^{\NN} : 
|(j,\CJ^f)|^d=N\in \NN\}.$
\end{center}

We first want to show that for a given $\delta>0$, if $(j, \CJ^f)\in \Psi$ is large enough,
\begin{equation}\label{pf AM Thm 1.3 eq 1}
\left|\NT(\sg\ZZ_S^d, (j,\CJ^f)A_0) - |(j,\CJ^f)|^d\right| = o( |(j,\CJ^f)|^{\frac 1 2 d + \delta})
\end{equation}
for almost all $\sg\in \mathcal F$, where $\mathcal F$ is a fundamental domain of $\SG/\Gamma$.

Choose a constant $\delta'$ such that $0<\delta'<2\delta/d$.
For each $\ell \in \NN$, define $\mathcal B_{\ell}$ as in Theorem \ref{Sch Lemma 2} with $\mathfrak F=\{A_0\}$ and $\psi(\ell)=2^{\ell\delta'}$.
Since $\delta'>0$, 
$$\sum_{\ell} \mu(\mathcal B_{\ell}) \le \sum_\ell 2^{-\ell\delta'}<\infty.$$

By the Borel-Cantelli lemma, for almost all $\sg \in \mathcal F$, there is $\ell_0>0$ such that if $\ell\ge \ell_0$,
for any $\CJ^f\in \prod_{p\in S_f} p^{\NN}$ with $|\CJ^f|^d \le 2^{\ell}$, 
\begin{equation}\label{pf AM Thm 1.3 eq 2}
\sum_{(N_1, N_2)\in K_{\ell}} R_{A_0}^2(N_1, N_2, \CJ^f)(\sg \ZZ_S^d) < (\ell+1)2^{\ell(1+\delta')}.
\end{equation}

Let $|(j,\CJ^f)|^d=N\in \NN$ and choose $\ell\in \NN$ such that  $2^{\ell-1}\le N < 2^{\ell}$. 
Note that $N$ can be expressed as the summation of at most $\ell$ number of $(N_2 - N_1)$, $(N_1, N_2) \in K_{\ell}$, mutually distinct.
Hence \eqref{pf AM Thm 1.3 eq 1} is obtained from \eqref{pf AM Thm 1.3 eq 2} and Cauchy inequality since
\[\begin{split}
&\left|\NT(\sg\ZZ_S^d, (j,\CJ^f)A_0) - |(j,\CJ^f)|^d\right|^2
< \ell (\ell+1) 2^{\ell(1+\delta')}\\
&\hspace{1.2in}< C (\log N)^2 N^{1+\delta'}
= o( N^{1+2\delta/d})=o(|(j,\CJ^f)|^{d+2\delta})
\end{split}\]
for some uniform constant $C>0$ and $N \gg 1$.

For arbitrary large $\T=(T_p)_{p\in S}$, let $\CJ^f=(T_p)_{p\in S_f}$ and choose $N$ and $j$, $j'$ such that
$N=|(j,\CJ^f)|^d \le |\T|^d < N+1=|(j',\CJ^f)|^d$. Then
\[
\NT(\sg\ZZ_S^d, (j, \CJ^f)A_0) - (N+1) 
\le \NT(\sg\ZZ_S^d, \T A_0) - |\T|^d
\le \NT(\sg\ZZ_S^d, (j', \CJ^f)A_0) - N.
\]

Since the upper and the lower bound is $o(|(j,\CJ^f)|^{d/2+\delta})$, the midterm is also $o(|\T|^{d/2+\delta})$.

\end{proof}

%%%%%%%%%%%%%%%%%%%%%%%%%%%%%%%%%%%%%%%%%%%%%%%%%
Before continuing to prove Theorem \ref{AM Thm 1.2}, let me introduce the following lemma, which is easily extendable to the $S$-arithmetic space. 
\begin{lemma}\cite[Lemma 2.3]{KY}\label{KY Lemma 2.3}
For any finite-volume sets $A_1 \subseteq A \subseteq A_2 \subset \QQ_S^d$ and any $S$-lattice $\Lambda \subset \QQ_S^d$, it follows that
\[\begin{split}
&\left|\NT(\Lambda, A) - \mu(A)\right|\\
&\hspace{0.5in}\le \max\left\{\left|\NT(\Lambda, A_1) - \mu(A_1)\right|,\left|\NT(\Lambda, A_2) - \mu(A_2)\right|\right\}
+ \mu(A_2-A_1).
\end{split}\]
\end{lemma}

\begin{proof}[Proof of Theorem~\ref{AM Thm 1.2}]
Fix a bounded interval $\I\subset \QQ_S$ and a compact set $\mathcal K \subset  \mathcal F$, where $\mathcal F$ is a fundamental domain of $\SG/\Gamma$.
Take 
$$A_{\q{}{}, \I, \T}=\q{-1}{}(\I)\cap B(\origin, \T),$$
where $B(\origin, \T)=\{\ov \in \QQ_S^d : \|\ov\|_p<T_p,\; p\in S\}$, so that
$A_{\q{}{}, \I, \T} \in \mathcal F_{prod}$, where $\mathcal F_{prod}$ is defined in Proposition \ref{Prop class F}, and
\[\begin{split}
\NT(\q{}{}, \I, \T)&=N(\ZZ_S^d, A_{\q{}{}, \I, \T}),\\
\VT(\q{}{}, \I, \T)&=\mu(A_{\q{}{}, \I, \T}).
\end{split}\]

As in the proof of Theorem \ref{AM Thm 1.3}, we will find a range of $(\gamma_\infty, \gamma_f)$ for which the set
\[
\bigcap_{\CJ_0}\hspace{-0.1in} \bigcup_{\scriptsize \begin{array}{c} \CJ :\\ \CJ \succeq \CJ_0\end{array}}
\hspace{-0.1in}\left\{\sg \in \mathcal K : \begin{array}{c}
\left|\NT(\q{\sg}{0}, \I, \T) - \VT(\q{\sg}{0}, \I, \T) \right|> (T_\infty)^{\frac {d-2} 2 +\gamma_\infty} |\T^f|^{\frac {d-2} 2 +\gamma_f}\\
\text{for some } \T \in [j_\infty, j_\infty+1)\times \CJ^f\end{array}\right\}
\]
is a null set, where $\CJ, \CJ_0$ are elements of $\NN \times \prod_{p\in S_f} p^{\NN}$. Put
\[
\mathcal C_{\CJ}=\left\{\sg \in \mathcal K : \begin{array}{c}
\left|\NT(\q{\sg}{0}, \I, \T) - \VT(\q{\sg}{0}, \I, \T) \right|> (T_\infty)^{\frac {d-2} 2 +\gamma_\infty} |\T^f|^{\frac {d-2} 2 +\gamma_f}\\
\text{for some } \T \in [j_\infty, j_\infty+1)\times \CJ^f\end{array}\right\}.
\]

For each $\CJ=(j_\infty, p_1^{j_1},\ldots, p_s^{j_s})$, 
let $\varepsilon_1=j_\infty^{-\alpha_\infty}|\CJ^f|^{-\alpha_f}$ and 
$\varepsilon_2=j_\infty^{-\beta_\infty}|\CJ^f|^{-\beta_f}$, where $\alpha_\infty$, $\alpha_f$, $\beta_\infty$ and $\beta_f$ are positive.

\vspace{0.1in}
%%%%%%%%%%%%%%%%%%%%%%%%%%%%%%%%%%%%%%%%%%%%%%%%
\noindent \emph{Covering of the space.\quad}
There is a constant $C(\mathcal K)>0$ (cf. \cite[Lemma 2.1]{KY}) such that for each $\CJ$, we can find a subset
$\mathcal Q=\mathcal Q(\mathcal K, \CJ)$ of $\mathcal K$ for which
\begin{enumerate}[(i)]
\item $\mathcal K \subseteq \bigcup\limits_{\sh \in \mathcal Q} \mathcal B(\sh, \varepsilon_1)$, where
%%%%%%%%%

\vspace{-0.1in}
\[\begin{split}
\mathcal B(\sh, \varepsilon_1)=
\Big(\left\{g_\infty \in \SL_d(\RR) : \|g_\infty\|_{op} \le 1+\varepsilon_1\right\}
\times\prod_{p\in S_f}G_p(\ZZ_{p})\Big).\sh;
\end{split}\]

\vspace{-0.2in}
\item $\# \mathcal Q(\mathcal K, \CJ) < C(\mathcal K) \varepsilon_1^{-\frac 1 2 (d+2)(d-1)}$.
\end{enumerate}

Here, $\|\cdot\|_{op}$ is the operator norm and $G_p(\ZZ_p)=\UL_d(\ZZ_p)$ if $\SG=\UL_d(\QQ_p)$ or $\SL_d(\ZZ_p)$ if $\SG=\SL_d(\QQ_p)$.
Note that the exponent $(d+2)(d-1)/2$ is the codimension of $\SO(d)$ in $\SL_d(\RR)$ since $\SG/\Gamma \simeq \SL_d(\RR)/\SL_d(\ZZ)\times \prod_{p\in S_f} G_p(\ZZ_p)$.

Observe that each 
\[
\ov \in \ZZ_S^d : \q{\sg\sh}{0}(\ov) \in \I, \|\ov\|_\infty < T_\infty \text{ and } \|\ov\|_{p} = T_p,\; \forall p\in S_f
\]
corresponds to 
\[
\ow(=\sg\ov) \in \sg\ZZ_S^d : \q{\sh}{0}(\ow) \in \I, \|\ow\|_\infty < T_\infty(1+\varepsilon_1) \text{ and } \|\ow\|_{p} = T_p,\; \forall p\in S_f.
\]

If we put $\T_1=(T_\infty(1-\varepsilon_1), \T^f)$ and $\T_2=(T_\infty(1+\varepsilon_1), \T^f)$,
we have
$$A_{\q{\sh}{0}, \I, \T_1}\subseteq A_{\q{\sg\sh}{0}, \I, \T}\subseteq A_{\q{\sh}{0}, \I, \T_2}.$$

By Lemma~\ref{KY Lemma 2.3}, since
\[
\mu(A_{\q{\sh}{0}, \I, \T_2}-A_{\q{\sh}{0}, \I, \T_1})\sim 
c_{\q{\sh}{0}} \mu(\I) 2(d-2) \varepsilon_1 |\T|^{d-2},
\]
it follows that
\[\begin{split}
&\left\{\sg \in \mathcal B(\sh, \varepsilon_1) : \begin{array}{c}
\left|\NT(\q{\sg\sh}{0},\I,\T)-\VT(\q{\sg\sh}{0},\I,\T)\right|>(T_\infty)^{\frac {d-2} 2 +\gamma_\infty} |\T^f|^{\frac {d-2} 2 +\gamma_f}\\
\text{for some } \T \in [j_\infty, j_\infty+1)\times \CJ^f \end{array}\right\}\\
&\subseteq\left\{\sg \in \mathcal B(\sh, \varepsilon_1) : \begin{array}{c}
\left|\NT(\sg\ZZ_S^d, A_{\q{\sh}{0},\I,\T})-\mu(A_{\q{\sh}{0},\I,\T})\right|>(T_\infty)^{\frac {d-2} 2 +\gamma_\infty} |\T^f|^{\frac {d-2} 2 +\gamma_f}\\
\text{for some } \T \in [j_\infty-\varepsilon_1, (j_\infty+1)+\varepsilon_1)\times \CJ^f \end{array}\right\},
\end{split}\]
if $\varepsilon_1|\T|^{d-2}< T_\infty^{(d-2)/2+\gamma_\infty} |\T^f|^{(d-2)/2+\gamma_f}$. Hence we need that
\begin{equation}\label{eq no 6}
\frac {d-2} 2 +\gamma_\infty > (d-2) - \alpha_\infty \quad\text{and}\quad \frac {d-2} 2 +\gamma_f > (d-2) - \alpha_f.
\end{equation}

\vspace{0.1in}
%%%%%%%%%%%%%%%%%%%%%%%%%%%%%%%%%%%%%%%%%%%%%%%%%
\noindent \emph{Transfer w.r.t. $\T$.\quad}
Let us divide the interval $[j_\infty-\varepsilon_1, (j_\infty+1)+\varepsilon_1)$ into 
$\left(\lfloor \frac {1+2\varepsilon_1} {\varepsilon_2} \rfloor+1\right)$ pieces.
For each $k=0,1,2,\ldots, \lfloor \frac {1+2\varepsilon_1} {\varepsilon_2} \rfloor$, observe that
\[\begin{split}
\mu\left(A_{\q{\sh}{0},\I,((j_\infty-\varepsilon_1)+(k+1)\varepsilon_2,\CJ^f)}-A_{\q{\sh}{0},\I,((j_\infty-\varepsilon_1)+k\varepsilon_2,\CJ^f)}\right)
\sim c_{\q{\sh}{0}} \mu(\I) \varepsilon_2 \;j_\infty^{d-3} |\CJ^f|^{d-2}.
\end{split}\]

Again, by Lemma~\ref{KY Lemma 2.3}, if
\begin{equation}\label{eq no 7}
\frac {d-2} 2 +\gamma_\infty > (d-3)-\beta_\infty \quad\text{and}\quad \frac {d-2} 2 +\gamma_f > (d-2) -\beta_f,
\end{equation}
we deduce that
\[\begin{split}
&\left\{\sg \in \mathcal B(\sh, \varepsilon_1) :
\begin{array}{c}
\left|\NT(\sg\ZZ_S^d, A_{\q{\sh}{0},\I,\T})-\mu(A_{\q{\sh}{0},\I,\T})\right|>(T_\infty)^{\frac {d-2} 2 +\gamma_\infty} |\T^f|^{\frac {d-2} 2 +\gamma_f}\\
\text{for some } \T \in [j_\infty-\varepsilon_1, (j_\infty+1)+\varepsilon_1)\times \CJ^f \end{array}\right\}\\
&\subseteq \bigcup_{k=0}^{\lfloor \frac {1+2\varepsilon_1} {\varepsilon_2}\rfloor}
\left\{\sg \in \mathcal B(\sh, \varepsilon_1) : \begin{array}{l}
\left|\NT(\sg\ZZ_S^d, A_{\q{\sh}{0},\I,\T})-\mu(A_{\q{\sh}{0},\I,\T})\right|\\
\hspace{1.35in}>(T_\infty)^{\frac {d-2} 2 +\gamma_\infty} |\T^f|^{\frac {d-2} 2 +\gamma_f}\\
\hspace{0.4in}\text{for } \T=\left((j_\infty-\varepsilon_1)+k\varepsilon_2, \CJ^f \right) \end{array}\right\}.
\end{split}\]

\vspace{0.1in}
As a result, by Chebyshev's inequality with Proposition \ref{Prop class F} (b) (see also \eqref{class F}),
it follows that
\[
m(\mathcal C_{\CJ}) \le C_{\mathcal F_{prod}}\: j_\infty^{-2\gamma_\infty} |\CJ^f|^{-2\gamma_f} \times \varepsilon_1^{-\frac 1 2 (d+2)(d-1)} \times
\left\lfloor \frac {1+2\varepsilon_1} {\varepsilon_2}\right\rfloor.
\]
Hence $\sum\limits_{\CJ\succeq \CJ_0} m(\mathcal C_{\CJ})\rightarrow 0$ as $\CJ_0\rightarrow \infty$ if
\begin{equation}\label{eq no 8}\begin{split}
-2\gamma_\infty+\alpha_\infty \Big(\frac 1 2 (d+2)(d-1)\Big)+ \beta_\infty &< -1 \;\text{and}\\
-2\gamma_f + \alpha_f\Big(\frac 1 2 (d+2)(d-1)\Big) + \beta_f &< 0.
\end{split}\end{equation}

By \eqref{eq no 6}, \eqref{eq no 7} and \eqref{eq no 8}, 
the set $\bigcap\limits_{\CJ_0}\bigcup\limits_{\CJ\succeq \CJ_0} \mathcal C_{\CJ}$ has measure zero if
\[
\gamma_\infty, \gamma_f > \frac {d-2} 2 \cdot \frac {d^2+d-1}{d^2+d+4}.
\]
\end{proof}

\end{document}